\documentclass{amsart}

\usepackage{latexsym}
\usepackage{amsmath,amsfonts,amssymb} 
\usepackage{graphicx}
\usepackage{tikz}

\usetikzlibrary{arrows,positioning}

\newcommand{\field}[1]{\mathbb{#1}}
\newcommand{\A}{\field{A}}
\newcommand{\C}{\field{C}}

\newcommand{\G}{\field{G}}
\newcommand{\N}{\field{N}}

\newcommand{\R}{\field{R}}

\newcommand{\Z}{\field{Z}}
\newcommand{\krn}{{\rm ker}\,}

\newtheorem{theorem}{Theorem}[section]
\newtheorem{proposition}[theorem]{Proposition}
\newtheorem{lemma}[theorem]{Lemma}
\newtheorem{corollary}[theorem]{Corollary}
\newtheorem{definition}[theorem]{Definition}

\newtheorem{example}[theorem]{Example}

\newtheorem{question}[theorem]{Question}

\begin{document}

\makeatletter	   
\makeatother     

\title{Canonical Factorization of the Quotient Morphism for an Affine $\G_a$-Variety}
\author{Gene Freudenburg} 
\dedicatory{To Mikhail Zaidenberg, mentor and friend, on the occasion of his seventieth birthday}
\date{\today} 
\maketitle

\pagestyle{headings}

 \begin{abstract} Working over a ground field of characteristic zero, 
 this paper studies the quotient morphism $\pi :X\to Y$ for an affine $\G_a$-variety $X$ with affine quotient $Y$. 
 It is shown that the degree modules associated to the $\G_a$-action give a uniquely determined sequence 
of dominant $\G_a$-equivariant morphisms, $X=X_r\to X_{r-1}\to\cdots\to X_1\to X_0=Y$, where 
 $X_i$ is an affine $\G_a$-variety and $X_{i+1}\to X_i$ is birational for each $i\ge 1$. This is the {\it canonical factorization} of $\pi$.
We give an algorithm for finding the degree modules associated to the given $\G_a$-action, and this yields the 
canonical factorization of the quotient morphism. The algorithm is applied to compute the canonical factorization for several examples, including the homogeneous $(2,5)$-action on $\A^3$. 
By a fundamental result of Kaliman and Zaidenberg, any birational morphism of affine varieties is an affine modification, and each mapping in these examples is presented
as a $\G_a$-equivariant affine modification. 
\end{abstract} 
 
 \section{Introduction}
 
We assume throughout that $k$ is a field of characteristic zero and $\G_a$ is the additive group of $k$. A $k$-affine $\G_a$-variety is an affine $k$-variety $X$ equipped with a
regular algebraic action of $\G_a$. The $\G_a$-actions on $X$ are in bijective correspondence with the locally nilpotent derivations 
$D$ of the coordinate ring $k[X]$, 
and if $Y$ is the categorical quotient of the given $\G_a$-action, then $k[Y]=\krn D$, the kernel of the derivation (see \cite{Freudenburg.06}). 
By a fundamental result of Winkelmann \cite{Winkelmann.03}, $Y$ is always quasi-affine.

Let $\pi :X\to Y$ be the quotient morphism for the given $\G_a$-action, and assume that $Y$ is affine. The main purpose of this paper is to define the canonical factorization of $\pi$ and to give methods for finding it. The canonical factorization gives critical information about the $\G_a$-action. In particular, the degree modules 
\[
\mathcal{F}_n=\krn D^{n+1}
\]
are used to define the canonical factorization, where $D$ is the locally nilpotent derivation of $k[X]$ induced by the $\G_a$-action. 
For each $n\ge 0$, $\mathcal{F}_n$ is a module over $\mathcal{F}_0=\krn D$, and $k[\mathcal{F}_N]=k[X]$ for some $N\ge 0$, since $k[X]$ is an affine ring. 
Consider the ascending chain of subalgebras:
\[
k[Y]=\mathcal{F}_0\subset k[\mathcal{F}_1]\subset\cdots\subset k[\mathcal{F}_N]=k[X]
\]
From this we obtain a uniquely determined sequence of integers $n_i$ with 
$k[\mathcal{F}_{n_i}]\varsubsetneqq k[\mathcal{F}_{n_{i+1}}]$.
If $X_i={\rm Spec}(k[\mathcal{F}_{n_i}])$ and $N=n_r$, then the corresponding sequence of morphisms
\[
X=X_r  \xrightarrow{\pi_{r-1}} X_{r-1}  \to\cdots\to X_2  \xrightarrow{\pi_1} X_1\xrightarrow{\pi_0} X_0=Y
\]
is the canonical factorization of $\pi$, and $r$ is the index of the action. {\it Lemma\,\ref{noether}} and {\it Lemma\,\ref{index-res}} show the following. 
\begin{enumerate}
\item For each $i$, $X_i$ is an affine $\G_a$-variety and the morphism $\pi_i$ is $\G_a$-equivariant.
\item The morphisms $\pi_1,...,\pi_{r-1}$ are birational.
\end{enumerate}

This factorization can be described in terms of affine modifications, which were introduced by Zariski as a tool in studying birational correspondences, and further developed by Davis
\cite{Davis.67,Zariski.43}. Geometrically,
an affine modification of the affine variety $X$ is a certain affine open subset $X'\subset X^*$, where $\beta :X^*\to X$ is a blow-up of $X$ (see {\it Sect.\,\ref{affine-mod}}). 
Kaliman and Zaidenberg were the first to systematically apply affine modifications to problems in affine algebraic geometry \cite{Kaliman.94, Kaliman.Zaidenberg.99}. They proved that any birational morphism of affine varieties is an affine modification (\cite{Kaliman.Zaidenberg.99} Thm.\,1.1). Moreover, in the case $X$ is a $\G_a$-variety, they give conditions as to when the action lifts to an affine modification $\beta :X'\to X$, i.e., $X'$ is a $\G_a$-variety and $\beta$ is $\G_a$-equivariant (\cite{Kaliman.Zaidenberg.99} Cor.\,2.3). This is called a $\G_a$-equivariant affine modification. Thus, each mapping in the canonical factorization for the quotient map of a $\G_a$-action is a $\G_a$-equivariant affine modification. 

The actions of $\G_a$ on affine spaces $X=\A^n$ are of particular interest. The $\G_a$-actions on the affine plane $X=\A^2$ were classified by Rentschler in 1968 
\cite{Rentschler.68}. This classification shows that the quotient morphism $\pi :X\to Y$ is of the form $X=Y\times\A^1$, where $\pi$ is projection on the first factor. 
Consequently, the index of any planar $\G_a$-action is one, whereas in dimension three, there are $\G_a$-actions of index greater than one. Although much 
is known about $\G_a$-actions on $\A^3$, their complete classification has not been achieved. 
It is known that, if $X\to\cdots\to X_1\to X_0=Y$ is the canonical factorization for a $\G_a$-action on $X=\A^3$, then $Y\cong\A^2$ (due to Miyanishi) and $X_1\cong\A^3$ (due to Bonnett and Daigle); see {\it Sect.\,\ref{dim-3}}. Thus, one is led to study affine threefolds $X_i$ of {\it sandwich type}, that is, those admitting birational $\G_a$-equivariant morphisms $\A^3\to X_i\to\A^3$. 

In {\it Sect.\,\ref{examples}}, our methods are applied to compute the canonical factorization for several examples, including actions on $\A^3$, where each of the birational maps in the factorization is presented as a $\G_a$-equivariant affine modification. It is hoped that canonical factorizations provide a new tool for making progress on the classification of $\G_a$-actions on $\A^3$.
Similarly, a great deal of work has been done on the classification of $\G_a$-surfaces (see, for example, \cite{Flenner.Zaidenberg.05}), and canonical factorizations should be of interest in this endeavor.

{\it Thm.\,\ref{main}} gives the theoretical basis for an algorithm to calculate the degree modules of a locally nilpotent derivation $D$ of a commutative $k$-domain $B$ in the case where
$\krn D$ is noetherian. As such, it provides a tool for calculating several related objects, which include the following.
\begin{enumerate}
\item We obtain a method to find the canonical factorization for the quotient morphism of a $\G_a$-action on an affine variety, assuming that the quotient is affine. 
\item The degree modules $\mathcal{F}_i$ determine the image ideals $D^i\mathcal{F}_i$, so the algorithm gives a way to find generators for these ideals. 
The plinth ideal $D\mathcal{F}_1$ is especially important.
\item The associated graded ring ${\rm Gr}_D(B)$ induced by the degree function of $D$ is determined by the degree modules $\mathcal{F}_n$, so the algorithm gives a way to find generators for this ring up to degree $n$ once $\mathcal{F}_n$ has been calculated. 
\item A basic problem of locally nilpotent derivations is to find generators for a given kernel $A\subset B$. In case $A$ itself admits a locally nilpotent derivation $\delta$, 
finding the degree modules for $\delta$ gives a generating set for $A$.
\end{enumerate}

\paragraph{\bf Preliminaries.} We assume throughout that $k$ is a field of characteristic zero. When working with varieties, we further assume that $k$ is algebraically closed. If $B$ is a commutative $k$-domain, then $B^{[n]}$ denotes the polynomial ring in $n$ variables over $B$. Given nonzero $f\in B$, $B_f$ denotes the localization $S^{-1}B$ for $S=\{ f^n\,\vert\, n\in\N\}$. 
If $X$ is an affine $k$-variety, then $k[X]$ is the coordinate ring of $X$. Given an ideal $I\subset k[X]$, $\mathcal{V}(I)\subset X$ is the zero set of $I$ in $X$. 
If $B=k[x_1,...,x_n]=k^{[n]}$, then for each $i$, $\partial_{x_i}$ denotes the partial derivative $\partial/\partial x_i$ of $B$ relative to the given system of coordinates. 
\medskip

\paragraph{\bf Acknowledgment.} The author gratefully acknowledges the assistance of Daniel Daigle and Shulim Kaliman, whose comments and advice led to several improvements in this paper. 


 \section{Locally Nilpotent Derivations and $\G_a$-Actions} 

Let $B$ be a commutative $k$-domain. 

A {\bf locally nilpotent derivation} of $B$ is a derivation $D:B\to B$ such that, for each $b\in B$, there exists $n\in\N$ (depending on $b$) such that $D^nb=0$. 
Let $\krn D$ denote the kernel of $D$. The set of locally nilpotent derivations of $B$ is denoted by ${\rm LND}(B)$. Note that any $D\in\text{LND}(B)$ is a $k$-derivation. 

The study of $\G_a$-actions on an irreducible affine $k$-variety $X$ is equivalent to the study of locally nilpotent derivations of the coordinate ring $k[X]$. 
In particular, the action induced by $D\in {\rm LND}(B)$ is given by the exponential map $\exp (tD)$, $t\in\G_a$, where the invariant ring $k[X]^{\G_a}$ equals $\krn D$. Conversely, every regular algebraic $\G_a$-action on $X$ is of this form. If $X$ is a $\G_a$-variety, then $X^{\G_a}$ denotes the set of fixed points of the $\G_a$-action on $X$, which is defined by the ideal $(DB)$ generated by the image of $D$. 

Assume that $D\in {\rm LND}(B)$ and $A=\krn D$. 
An ideal $I\subset B$ is an {\bf integral ideal} for $D$ if $DI\subset I$. 
A {\bf slice} for $D$ is any $s\in B$ such that $Ds=1$. Note that $D$ has a slice if and only if the mapping $D:B\to B$ is surjective.  In this case, the {\bf Dixmier map} 
$\pi_s :B\to A$ given by $\pi_s(b)=\sum_{i\ge 0}\frac{(-1)^i}{i!}D^ib\,s^i$ is a surjective homomorphism of $k$-algebras. 

A {\bf local slice} for $D$ is any $r\in B$ such that $D^2r=0$ but $Dr\ne 0$. 
If $f=Dr$, then:
\[
B_f=A_f[r]=A_f^{[1]}
\]
The degree function $\deg _r$ on $B_f$ restricts to $B$, and this restricted function is denoted by $\deg_D$. Any other local slice of $D$ determines the same degree function, and if $E\in {\rm LND}(B)$ and $\krn E=A$, then $D$ and $E$ determine the same degree function. So $\deg_D$ is completely determined by the subring $A$.

Since $A=\{ b\in B\,\vert\, \deg_Db\le 0\}$ and $\deg b\ge 0$ for $b\ne 0$, $A$ is {\bf factorially closed} in $B$, 
meaning that $a,b\in A$ whenever $a,b\in B$ and $ab\in A\setminus\{ 0\}$. 
It follows that:
\begin{equation}\label{krn-ideal}
gB\cap A=gA \quad {\rm for\,\, all}\quad g\in A
\end{equation}

\begin{proposition} {\bf (Generating Principle)} Let $r\in B$ be a local slice of $D$ with $f=Dr$. 
 Suppose that $S\subset B$ is a subalgebra satisfying:
 \begin{enumerate}
 \item[(i)] $A[r]\subset S\subset B$
 \item[(ii)] $fB\cap S=fS$
 \end{enumerate}
 Then $S=B$.
 \end{proposition}
 
 \begin{proof} Let $b\in B$ be given. Since $r$ is a local slice, there exists an integer $n\ge 0$ such that $f^nb\in A[r]$. By hypothesis (i), $f^nb\in S$. 
 Since $B$ is a domain, repeated application of hypothesis (ii) shows that 
 $f^nB\cap S=f^nS$. Therefore, $f^nb=f^ns$ for some $s\in S$. Since $B$ is a domain, it follows that $b=s$. 
 \end{proof}
 
 If $M\subset B$ is an $A$-submodule, then $M$ is {\bf factorially closed} in $B$ if $\alpha ,\beta\in M$ whenever $\alpha ,\beta\in B$ and $\alpha\beta\in M\setminus\{ 0\}$. 
 We also need the following.

 \begin{definition} Given a non-empty set $V\subset B$, $V$ is a {\bf $D$-set} if the restriction $\deg_D :V\to\N\cup\{-\infty\}$ is injective.
 Let $M\subset B$ be a free $A$-module. A {\bf $D$-basis} of $M$ is a basis which is a $D$-set.
\end{definition}

The following two lemmas are obvious but useful.
\begin{lemma} Let $V\subset B$ be a $D$-set.
\begin{itemize} 
\item [{\bf (a)}] The elements of $V$ are linearly independent over $A$.
\item [{\bf (b)}] If $b\in B$ and $\deg_Db>\deg_Dv$ for all $v\in V$, then $\cup_{i\ge 0}Vb^i$ is a $D$-set.
\end{itemize}
\end{lemma}

\begin{lemma}\label{lemma2} Let $M\subset B$ be a finitely generated $A$-module, $M=\sum_{1\le i\le n}Aw_i$, and let $b\in B$ satisfy:
\[
\deg_Db>\max\{ \deg_Dw_i\, |\, 1\le i\le n\}
\]
Define the $A$-module $N=\sum_{i\ge 0}Mb^i$. 
\begin{enumerate}
\item [{\bf (a)}] If $M$ is a free $A$-module, then $N$ is a free $A$-module.
\item [{\bf (b)}] If $M$ admits a $D$--basis $\{ v_1,...,v_m\}$, then $N$ admits a $D$-basis given by:
\[
\{ v_ib^j\, |\, 1\le i\le m\, ,\, j\ge 0\}
\]
\end{enumerate}
\end{lemma}
 
The reader is referred to \cite{Freudenburg.06} for further details about locally nilpotent derivations. 
 
 
 \section{Affine Modifications}\label{affine-mod}  
 
 In this section, we follow the notation and terminology of \cite{Kaliman.Zaidenberg.99}
 
An {\bf affine triple} over $k$ is of the form $(B,I,f)$, where $B$ is an affine $k$-domain, $I\subset B$ is a nonzero ideal and $f\in I$, $f\ne 0$.
For such a triple, define
\[
f^{-1}I=\{ g\in B_f\,\vert\, fg\in I\}
\]
where $B_f$ is the localization of $B$ at $f$. 
\begin{definition} Given the affine triple $(B,I,f)$, the ring $B[f^{-1}I]$ is the {\bf affine modification} of $B$ {\bf along} $f$ with {\bf center} $I$.
\end{definition}

Let $b_1,...,b_s\in B$ be such that $I=(b_1,...,b_s)$. Then $B[f^{-1}I]=B[b_1/f,...,b_s/f]$, and $B[f^{-1}I]$ is an affine domain. 
Let $X={\rm Spec}(B)$ and $X_{(I,f)}={\rm Spec}(B[f^{-1}I])$. 
\begin{definition} $X_{(I,f)}$ is the {\bf affine modification} of $X$ {\bf along} $f$ with {\bf center} $I$. 
The morphism $p:X_{(I,f)}\to X$ induced by the inclusion $B\subset B[f^{-1}I]$ is the {\bf associated morphism} for the affine modification. 
\end{definition}
Since $B_f= B[f^{-1}I]_f$, we see that
the associated morphism $p: X_{(I,f)}\to X$ is birational, and that the restriction of $p$ to the set $\{ f\ne 0\}$ is an isomorphism. 
The {\bf exceptional divisor} $E$ of $X_{(I,f)}$ is defined by the ideal $IB[f^{-1}I]$. 

Our main interest is in the following fact, due to Kaliman and Zaidenberg. 

\begin{theorem}\label{KZ1} {\rm (\cite{Kaliman.Zaidenberg.99}, Thm.\,1.1)} Any birational morphism of affine varieties is the associated morphism of an affine modification. 
\end{theorem} 

\subsection{Principal Affine Modifications}

\begin{definition} Let $(B,I,f)$ be an affine triple, and let $B'=B[f^{-1}I]$ be the induced affine modification. $B'$ is a {\bf principal} affine modification of $B$ if and only if $B'$ is a principal ring extension of $B$, that is, $B'=B[r]$ for some $r\in{\rm frac}(B)$. In this case, $X'={\rm Spec}(B')$ is a {\bf principal} affine modification of $X={\rm Spec}(B)$. 
\end{definition}

\begin{lemma}\label{principal} Let $(B,I,f)$ be an affine triple, and let $B'=B[f^{-1}I]$ be the induced affine modification. Then $B'$ is principal if and only if 
there exist $g\in B$ and $n\ge 0$ such that $B'=B[(f^n)^{-1}J]$, where $J=f^nB+gB$. 
\end{lemma}

\begin{proof} Suppose that $g\in B$ and $n\ge 0$ are such that 
$B'=B[(f^n)^{-1}J]$ for the ideal $J=f^nB+gB$. Then $B'=B[\frac{g}{f^n}]$, which is principal.

Conversely, if $B'=B[r]$, then, by definition, there exists $n\ge 0$ such that $f^nr\in B$. Set $g=f^nr$ and $J=f^nB+gB$. Then $B[(f^n)^{-1}J]=B'$.
\end{proof}

\subsection{Composing Affine Modifications}

Let $(B,I,f)$ and $(B,J,g)$ be affine triples. Form the affine modifications:
\[
B_1=B[f^{-1}I]\,\, ,\,\, B_2=B[g^{-1}J] \,\, ,\,\, B'=[(fg)^{-1}IJ]
\]
Define $X_1={\rm Spec}(B_1)$, $X_2={\rm Spec}(B_2)$ and $X'={\rm Spec}(B')$. It is easy to check that:
\[
B'=B_1[g^{-1}JB_1]=B_2[f^{-1}IB_2]
\]
Therefore, the following diagram of affine modifications commutes.
\medskip
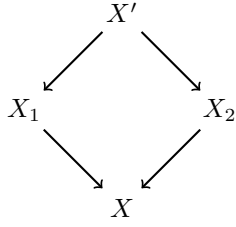
\begin{center}
\begin{tikzpicture}[scale=1.3]
\node at (0,1){$X'$};
\node at (1,0){$X_2$};
\node at (-1,0){$X_1$};
\node at (0,-1){$X$};
 \draw[thick,<-] (-.8,.2) -- (-.2,.8);
 \draw[thick,->] (.2,.8) -- (.8,.2);
 \draw[thick,->] (.8,-.2) -- (.2,-.8);
 \draw[thick,->] (-.8,-.2) -- (-.2,-.8);
\end{tikzpicture}
\end{center}

\subsection{Equivariant Affine Modifications}

Assume that $X$ is endowed with a $\G_a$-action. 
\begin{definition} The affine modification $X_{(I,f)}$ of $X$ is $\G_a$-{\bf equivariant} if the $\G_a$-action on $X$ lifts to $X_{(I,f)}$, i.e., $X_{(I,f)}$ admits a $\G_a$-action for which the associated morphism is $\G_a$-equivariant.
\end{definition}

Note that, if $X_{(I,f)}$ admits a $\G_a$-action for which the associated morphism is $\G_a$-equivariant, then the action is uniquely determined.

Suppose that $D\in {\rm LND}(B)$, $f\in\krn D$ is nonzero, and $I\subset B$ is a nonzero integral ideal for $D$ (i.e., $DI\subset I$). 
Let $D'$ be the extension of $D$ to $B_f$. Then $D'$ is locally nilpotent and:
\[
D'(f^{-1}I)=f^{-1}D'I\subset f^{-1}I
\]
It follows that $D'$ restricts (and $D$ extends) to $B[f^{-1}I]$. 
We have thus shown:

\begin{theorem}\label{KZ2} {\rm (\cite{Kaliman.Zaidenberg.99}, Cor.\,2.3)} Let $\rho :\G_a\times X\to X$ be a $\G_a$-action and let $X_{(I,f)}$ be an affine modification of $X$
along $f$ with center $I$. If $f\in k[X]^{\G_a}$ and $\rho$ restricts to an action on $\mathcal{V}(I)$, then $X_{(I,f)}$ is a $\G_a$-equivariant affine modification. 
\end{theorem}

\begin{example} {\rm 

Consider the affine plane $X=\A^2$ and its coordinate ring $B=k[x,y]=k^{[2]}$. Let $L,M\subset X$ be the lines defined by $x=0$ and $y=0$, respectively, and let $P$ be their intersection. The locally nilpotent derivation $D=x\frac{\partial}{\partial y}$ of $B$ induces a $\G_a$-action on $X$, and $L=X^{\G_a}$. The one-dimensional orbits are lines $x=x_0$ for $x_0\ne 0$. 

Let $\beta : X^*\to X$ be the blow-up of $X$ at $P$, let $E\subset X^*$ be the exceptional divisor over $P$, and let $L^*,M^*\subset X^*$ be the strict transforms of $L$ and $M$, respectively. 
Let $U\subset X^*$ be the open set $U=X^*\setminus L^*$ and let $E_0=E\cap U$. 
Then $U\cong \A^2$, and the composition $U\hookrightarrow X^*\xrightarrow{\beta} X$ is a birational endomorphism of $\A^2$. 

As an affine modification, the coordinate ring $k[U]$ is $B[x^{-1}I]=k[x,\frac{y}{x}]$, where $I\subset B$ is the ideal $I=xB+yB$ defining $P$. Note that $DI\subset I$. 
The morphism $U\to X$ is defined by the inclusion 
$k[x,y]\hookrightarrow k[x,\frac{y}{x}]$, with exceptional divisor $E_0$. $D$ extends to the derivation $D'$ on $k[x,\frac{y}{x}]$ defined by $D'x=0$, $D'(\frac{y}{x})=1$.
Thus, the $\G_a$-action on $X$ lifts to a free $\G_a$-action on $U$, and the associated morphism $U\to X$ is $\G_a$-equivariant.

The situation is depicted in {\it Fig.\,1}. }
\end{example}

\begin{figure}[b]\label{Fig1}

\begin{tikzpicture}[scale=1.3]
    \draw[thick] (-1,0) -- node[left]{$L=X^{\G_a}\,\,$} (.2,1.2) (1,0) -- node[right]{$\,\,\, M$} (-.2,1.2);
    \filldraw (0,1) circle (1.5pt);
     \node[below] at (0,-.2){$X\cong\A^2$};
    \draw[thick,<-] (1.4,.5) -- (2.1,.5);
    \node[above] at (1.75,.5){$\beta$};
    \draw[thick] (3,0) -- node[left]{$L^*$} (3,1.2) (5,0) -- node[right]{$M^*$} (5,1.2) (2.8,1) -- node[above]{$E$} (5.2,1);
 
    \node[below] at (4,-.2){$X^*$};
     \draw[thick,left hook->] (6.5,.5) -- (5.8,.5);
     \draw[thick] (9,0) -- node[right]{$M^*$}  (9,1.2) (6.8,1) -- node[above]{$E_0$} (9.2,1);
     \draw[thick,dashed] (7,0) -- (7,1.2);
   
    \node[below] at (8,-.2){$U\cong\A^2$};
    \node[above] at (0,1.2){$P$};
\end{tikzpicture}

\caption{A birational $\G_a$-endomorphism of $\A^2$}
\end{figure}
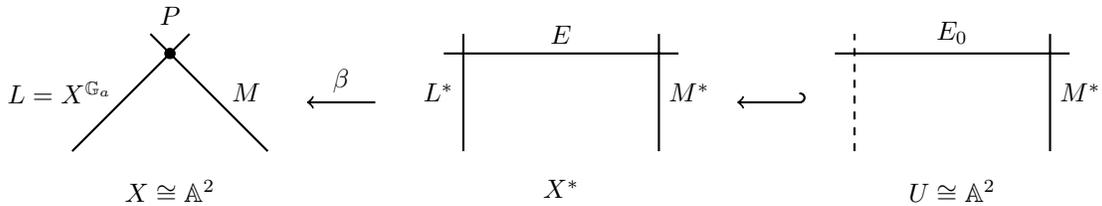
 
 
 \section{Degree Modules}\label{algorithm}
 
 Assume that $B$ is a commutative $k$-domain and that $A\subset B$ is a subalgebra such that $A=\krn D$ for some nonzero $D\in {\rm LND}(B)$. Our primary interest is in the following three related objects. 
\medskip
 
 \begin{enumerate}
 \item  The ascending $\N$-filtration of $B$ by $A$-modules given by:
\[
B=\bigcup_{n\ge 0}\mathcal{F}_n\quad {\rm where}\quad \mathcal{F}_n=\krn D^{n+1}=\{ f\in B\, |\, \deg_Df\le n\}
\] 
The modules $\mathcal{F}_n$ are the {\bf degree modules} associated to $D$. 
\medskip
\item The descending $\N$-filtration of $A$ by ideals given by:
\[
A=I_0\supset I_1\supset I_2\supset\cdots \quad{\rm where}\quad I_n=A\cap D^nB=D^n\mathcal{F}_n
\]
The ideals $I_n$ are the {\bf image ideals} associated to $D$. The  {\bf plinth ideal} for $D$ is ${\rm pl}(D)=I_1$.
\medskip
\item The ring
\[
{\rm Gr}_D(B) = \bigoplus_{n\ge 0}I_n\cdot t^n \subset A[t]\cong A^{[1]}
\]
is the {\bf associated graded ring} defined by $D$. 
\end{enumerate}
\medskip

Observe the following.
 \begin{itemize}
  \item [(a)] Each $\mathcal{F}_n$ is a factorially closed $A$-submodule of $B$.
 \item [(b)] The definition of $\mathcal{F}_n$ depends only on $A$, not on the particular derivation $D$. 
 \item [(c)] Given integers $n,i$ with $1\le i\le n$, the following sequence of $A$-modules is exact.
\[
0\to \mathcal{F}_{i-1}\hookrightarrow \mathcal{F}_n \xrightarrow{D^i} D^i\mathcal{F}_n\to 0
\]
In particular, $I_n\cong \mathcal{F}_n/\mathcal{F}_{n-1}$.
\item [(d)] Let $r\in B$ be a local slice for $D$. For each $n\ge 0$, define the submodule $\mathcal{G}_n(r)\subset\mathcal{F}_n$ by:
\[
\mathcal{G}_n(r)=A[r]\cap\mathcal{F}_n = A\oplus Ar\oplus \cdots \oplus Ar^n
\]
If $r$ is a slice for $D$, then $\mathcal{G}_n(r)=\mathcal{F}_n$ for each $n\ge 0$. 
\end{itemize}

\begin{lemma}\label{noether} If $A$ is a noetherian ring, then $\mathcal{F}_n$ is a noetherian $A$-module for each $n\ge 0$.
\end{lemma}

\begin{proof} $\mathcal{F}_0=A$ is noetherian by hypothesis. 
Given $n\ge 1$, assume by induction that $\mathcal{F}_m$ is noetherian for $0\le m\le n-1$. 
By the inductive hypothesis, $\mathcal{F}_{i-1}$ and $\mathcal{F}_{n-i}$ are noetherian for $1\le i\le n$. Therefore, 
the submodule $D^i\mathcal{F}_n$ of $\mathcal{F}_{n-i}$ is also noetherian. Since the sequence (3) above is exact, and 
since $\mathcal{F}_{i-1}$ and $D^i\mathcal{F}_n$ are noetherian, it follows that $\mathcal{F}_n$ is noetherian. 
\end{proof}

\begin{example} {\rm Assume that $A$ is a noetherian ring. Then there exist an integer $m\ge 1$ and $a_1,...,a_m\in A$ such that ${\rm pl}(D)=a_1A+\cdots +a_mA$. 
Let $r_1,...,r_m\in\mathcal{F}_1$ be such that $Dr_i=a_i$ for $1\le i\le m$. Given $s\in\mathcal{F}_1$, write $Ds=c_1a_1+\cdots +c_ma_m$ for $c_i\in A$. Then 
$s-(c_1r_1+\cdots c_mr_m)\in A$. It follows that:
\[
\mathcal{F}_1=A+Ar_1+\cdots +Ar_m
\]
See \cite{Alhajjar.15}, Lemma 2.2.}
\end{example}

Fix a local slice $r\in B$ and integer $n\ge 1$, and set $f=Dr\in A$. Suppose that $M_0$ is an $A$-submodule of $B$ such that $\mathcal{G}_n(r)\subset M_0\subset\mathcal{F}_n$. 
Inductively, define the ascending chain 
 $M_0\subset M_1\subset M_2\subset \cdots$ of $A$-submodules of $B$ by:
 \[
 M_i=\{ h\in B\,\vert\, fh\in M_{i-1}\} = \{ h\in B\,\vert\, f^ih\in M_0\} \quad (i\ge 1)
 \]
 Then $fM_{i+1}\subset M_i\subset M_{i+1}\subset \mathcal{F}_n$ for each $i\ge 0$, since $\mathcal{F}_n$ is factorially closed. 
 \begin{lemma}\label{filter} $\mathcal{F}_n=\cup_{i\ge 0}M_i$
 \end{lemma}
 
 \begin{proof}
It must be shown that, to each $h\in \mathcal{F}_n$, there exists $s\ge 0$ such that $h\in M_s$. Let $s\ge 0$ be such that $f^sh\in A[r]$.
Then $f^sh\in A[r]\cap\mathcal{F}_n=\mathcal{G}_n(r)\subset M_0$. 
By definition of the modules $M_i$, we see that $h\in M_s$, and the lemma is proved. 
\end{proof}
 
 \begin{theorem}\label{main} The following conditions are equivalent. 
 \begin{enumerate}
 \item  $fB\cap M_s=fM_s$ for some $s\ge 0$. 
 \item  $M_s=M_{s+1}$ for some $s\ge 0$.
 \item The ascending chain $M_0\subset M_1\subset M_2\subset\cdots $ stabilizes.
 \item $\mathcal{F}_n=M_s$ for some $s\ge 0$.
 \end{enumerate}
 If $A$ is a noetherian ring, then these conditions are valid. 
 \end{theorem}
 
 \begin{proof} 
(1) $\Leftrightarrow$ (2): This follows by definition of the modules $M_i$.

(2) $\Rightarrow$ (3): Assume that $M_s=M_{s+1}$ for some $s\ge 0$. If $h\in M_{s+2}$, then:
\[
fh\in M_{s+1}=M_s \quad\Rightarrow\quad h\in M_{s+1}=M_s
\]
Therefore, $M_s=M_{s+2}$. By induction, we obtain that $M_s=M_S$ for all $S\ge s$. 
 
(3) $\Rightarrow$ (4): Assume that, for some $s\ge 0$, $M_s=M_S$ for all $S\ge s$. By {\it Lemma\,\ref{filter}}, it follows that $\mathcal{F}_n=\cup_{i\ge 0}M_i = M_s$. 

(4) $\Rightarrow$ (2):  Assume that $\mathcal{F}_n=M_s$ for some $s\ge 0$. Then $M_{s+1}\subset \mathcal{F}_n=M_s\subset M_{s+1}$ implies that $M_s=M_{s+1}$.

We have thus shown that conditions (1)-(4) are equivalent. Assume that $A$ is a noetherian ring.
By {\it Lemma\,\ref{noether}}, there exists a finite module basis $\{ z_1,...,z_t\}$ for $\mathcal{F}_n$, where $t\ge 1$. 
By {\it Lemma\,\ref{filter}}, there exists $s\ge 0$ such that $\{ z_1,...,z_t\}\subset M_s$. Therefore, $\mathcal{F}_n=M_s$, and condition (4) is validated. 
\end{proof}

{\it Theorem\,\ref{main}} gives the theoretical basis for an algorithm to calculate the degree modules $\mathcal{F}_n$ in the case where $A$ is noetherian. 
Suppose that $\{ x_1,...,x_m\}$ is a set of module generators for $M_i$ and let $\{X_1,...,X_m\}$ be a basis for the free $A$-module of rank $m$. 
Define $\rho :A^m\to M_i$ by $\rho (X_j)=x_j$. 
Then $K:=\rho^{-1}(fB\cap M_i)$ is a submodule of $A^m$. Since $A$ is noetherian, $K$ is finitely generated. Generators for $K$ can be calculated by standard methods.
Suppose that 
$\{ Y_1,...,Y_l\}$ is a set of generators for $K$, and let $s_1,...,s_l\in B$ be such that $\rho (Y_j)=fs_j$. Then $M_{i+1}=M_i+As_1+\cdots +As_l$.


\section{Degree Resolutions}

We continue the notation and assumptions of the preceding section. 

\subsection{The Subrings $k[\mathcal{F}_n]$}

Define subrings $B_i=k[\mathcal{F}_i]\subset B$, $i\ge 0$. Then $B_0=A$ and $B_i\subset B_{i+1}$ for $i\ge 0$.
If $B$ is $G$-graded by an abelian group $G$ and $A$ is a $G$-graded subalgebra, then each $\mathcal{F}_i$ is a $G$-graded submodule and each $B_i$ is a $G$-graded subalgebra.

\begin{lemma}\label{index-res} Let $i$ be an integer, $i\ge 0$.
\begin{itemize} 
\item [{\bf (a)}] $D$ restricts to $D_i:B_i\to B_i$, where $A=\krn D_i$
\item [{\bf (b)}] If $i\ge 1$, then ${\rm frac}(B_i)={\rm frac}(B)$
\end{itemize}
\end{lemma}

\begin{proof} Given $i\ge 1$, the definition of $\mathcal{F}_i$ implies that $D(\mathcal{F}_i)\subset\mathcal{F}_{i-1}\subset\mathcal{F}_i$. Since $D$ restricts to a generating set for $B_i$, it follows that $D$ restricts to $B_i$, and part (a) is confirmed.

Part (b) follows from the observation that $S^{-1}B_i=S^{-1}B=S^{-1}A[r]$ for $S=A\setminus\{ 0\}$ and some $r\in \mathcal{F}_1$.
\end{proof}

\subsection{Degree Resolutions for Affine Rings}

Assume that $B$ is an affine $k$-domain. In this case, $B_N=B$ for some $N\ge 0$. 
It is possible that $B_i=B_{i+1}$ for some $i$. Let $n_i$, $0\le i\le r$, be the unique subsequence of $0,1,...,N$ such that:
\[
 \{ B_0,...,B_N\}=\{ B_{n_0},...,B_{n_r}\}   \quad{\rm and}\quad  B_{n_i-1}\subsetneqq B_{n_i}\subsetneqq B_{n_{i+1}} \,\,{\rm for}\,\, 1\le i< r
\]
Note that, when $D\ne 0$, $n_0=0$, $n_1=1$ and $B_{n_r}=B$. Let $\mathcal{N}_B(A)=\{0,1,n_2,...,n_r\}$. 
Both the integer $r$ and the sequence of subrings
\begin{equation}\label{resolution}
A=B_0\subset B_1\subset B_{n_2}\subset\cdots\subset B_{n_r}=B
\end{equation}
are uniquely determined by $A$. 
\begin{definition} The sequence of inclusions {\rm  (\ref{resolution})} is the {\bf degree resolution} of $B$ over $A$. 
The integer $r$ is the {\bf index} of $A$ in $B$, denoted ${\rm index}_B(A)$; we also say that $r$ is the {\bf index} of $D$.
\end{definition}

We make the following observations.
\begin{enumerate}
\item [(a)] ${\rm index}_B(A)+1=\vert\mathcal{N}_B(A)\vert$
\item [(b)] ${\rm index}_{B_{n_i}}(A)=i+1$
\item [(c)] ${\rm index}_B(A)=0$ if and only if $A=B$ if and only if $D=0$
\item [(d)] If $D$ has a slice, then ${\rm index}_B(A)=1$
\end{enumerate}

Let $R\subset B$ be an affine subring such that $D$ restricts to $R$. The induced filtration of $R$ is $R=\bigcup_{i\ge 0}R\cap\mathcal{F}_i$ 
and if $R_{n_i}=R\cap B_{n_i}$ for $n_i\in\mathcal{N}_B(A)$, then the degree resolution of $R$ over $R\cap A$ is a refinement of the sequence:
\[
R\cap A=R_0\subset R_1\subset R_{n_2}\subset \cdots\subset R_{n_r}=R
\]
Therefore, $\mathcal{N}_R(R\cap A)\subset\mathcal{N}_B(A)$ and ${\rm index}_R(R\cap A)\le {\rm index}_B(A)$.

\begin{example}\label{dim-2} {\rm Let $B=k^{[2]}$ and let $D\in {\rm LND}(B)$ be nonzero. By Rentschler's Theorem, there exist
$x,y\in B$ such that $A=\krn D=k[x]$, $Dy\in k[x]$ and $B=k[x,y]$. See \cite{Freudenburg.06}, Thm.\,4.1. 
Therefore, $\mathcal{F}_n=\mathcal{G}_n(y)=A\oplus Ay\oplus \cdots \oplus Ay^n$ for each $n\ge 0$. In particular, every nonzero element of ${\rm LND}(B)$ has index one. }
\end{example}

\begin{example}\label{dim-3} {\rm
Let $B=k[x,y,z]=k^{[3]}$. Define $P,Q\in B$ by $Q=xz+y^2$ and $P=y+Q^2$, and 
define $D\in {\rm LND}(B)$ by:
\[
D=P_z\partial_y-P_y\partial_z = 2xQ\,\partial_y-(1+4yQ)\,\partial_z
\]
Then $A=\krn D=k[x,P]$ and $DQ=x$. Define $J\subset\Z^2$ by:
\[
J=\{ (i,j)\,|\, 0\le i\le 3\, ,\, j\ge 0\}
\]
Given $n\ge 0$, define $J_n\subset J$ by:
\[
J_n=\{ (i,j)\in J\, |\, i+4j\le n\}
\]
We will show:
\begin{equation}\label{dim3}
B= \bigoplus_{(i,j)\in J}A\, Q^iz^j \quad {\rm and}\quad \mathcal{F}_n = \bigoplus_{(i,j)\in J_n}A\, Q^iz^j 
\end{equation}
Since $y\in k[P,Q]$ we have $B=k[x,P,Q,z]=A[z,Q]$. In addition, the equality 
\[
xz=Q-(P-Q^2)^2
\]
shows that $Q$ is integral of degree 4 over $A[z]\cong_AA^{[1]}$. Therefore:
\[
B=A[z,Q] = A[z]\oplus A[z]Q\oplus A[z]Q^2\oplus A[z]Q^3 = \bigoplus_{(i,j)\in J}A\, Q^iz^j
\]
This shows the first equality of (\ref{dim3}). Since $\deg_DQ=1$ and $\deg_Dz=4$, the degrees $\deg_D(Q^iz^j)$ for $(i,j)\in J$ are distinct and $\{Q^iz^j\,\vert\, (i,j)\in J\}$ is a $D$-basis for $B$, which implies the second equality of (\ref{dim3}). 
Therefore, $\mathcal{N}_B(A)=\{ 0,1,4\}$ and ${\rm index}_B(A)=2$. 
}
\end{example}

\begin{example}\label{dim-4} {\rm Let $B=k[x_1,x_2,y_1,y_2]=k^{[4]}$ and define $T\in {\rm LND}(B)$ by:
\[
T=x_1\partial_{y_1}+x_2\partial_{y_2}
\]
Then $A=\krn T=k[x_1,x_2,g]$ where $g=x_1y_2-x_2y_1$. Since $y_1$ and $y_2$ are local slices, $B=B_1=k[\mathcal{F}_1]$, and ${\rm index}_B(A)=1$. 
We calculate $\mathcal{F}_1$.

Define $M\subset\mathcal{F}_1$ by $M=A+Ay_1+Ay_2$. 
Suppose that $x_1w\in M$ for $w\in\mathcal{F}_1$, and write $x_1w=a_0+a_1y_1+a_2y_2$ for $a_i\in A$. Then $x_1Dw=a_1x_1+a_2x_2$, which implies that 
$a_2\in x_1B\cap A=x_1A$ and $a_0+a_1y_1\in x_1B$. 

Let $p:B\to B/x_1B$ be the canonical surjection, let $\bar{b}=p(b)$ for $b\in B$, and let $\bar{A}=p(A)=k[\bar{x_2},\bar{g}]$, noting that $\bar{g}+\bar{x}_2\bar{y}_1 = 0$. 
If $\bar{A}X\oplus\bar{A}Y$ is the free $\bar{A}$-module of rank 2, then:
\[
\bar{A}+\bar{A}\bar{y}_1 = \bar{A}X\oplus\bar{A}Y/\bar{A}(\bar{g}X+\bar{x}_2Y)
\]
Therefore, $\bar{a}_0+\bar{a}_1\bar{y}_1=0$ implies $a_0+a_1y_1=\alpha (g+x_2y_1)=\alpha x_1y_2$ for some $\alpha\in A$. So $w\in M$ and $M=\mathcal{F}_1$. 
In addition, the plinth ideal $D\mathcal{F}_1$ equals $x_1A+x_2A$. 
Note that, unlike in the preceding examples, $\mathcal{F}_1$ is not a free $A$-module. 
}
\end{example}


\section{Canonical Factorizations}\label{canonical}

We continue the notation and assumptions of the preceding section, 
with the added assumptions that $k$ is algebraically closed, and that both $A$ and $B$ are $k$-affine. In this case, the geometric content of {\it Lemma\,\ref{index-res}} is as follows. 

Let $X={\rm Spec}(B)$ and $Y={\rm Spec}(A)$, and let $\pi :X\to Y$ be the quotient map for the $\G_a$-action on $X$ determined by $D$. 
By {\it Lemma\,\ref{noether}},  $B_{n_i}$ is affine for each $n_i\in\mathcal{N}_B(A)$. Define $X_i={\rm Spec}(B_{n_i})$,
$0\le i\le r$. 

For $0\le i\le r-1$, the inclusion $B_{n_i}\to B_{n_{i+1}}$ induces a dominant $\G_a$-equivariant morphism $\pi_i:X_{i+1}\to X_i$ which is birational if $i\ne 0$.
Therefore, $\pi$ factors into the uniquely determined sequence of dominant $\G_a$-equivariant morphisms
\begin{equation}\label{quotient}
X=X_r  \xrightarrow{\pi_{r-1}} X_{r-1}  \to\cdots\to X_2  \xrightarrow{\pi_1} X_1\xrightarrow{\pi_0} X_0=Y
\end{equation}
where each morphism $\pi_{r-1},...,\pi_1$ is birational. 
\begin{definition} The sequence of mappings {\rm (\ref{quotient})} is the {\bf canonical factorization} of the quotient morphism $\pi$ for the $\G_a$-action determined by $D$. The integer $r$ is the {\bf index} of the $\G_a$-action. 
\end{definition}
From {\it Thm.\,\ref{KZ1}} and {\it Thm.\,\ref{KZ2}}, we conclude that the maps $\pi_1,...,\pi_{r-1}$  in the canonical factorization (\ref{quotient}) form a sequence of $\G_a$-equivariant affine modifications. 
Regarding fixed points, note that $\pi_{i-1}(X_i^{\G_a})\subset X_{i-1}^{\G_a}$ for each $i$ with $1\le i\le r$.
Note also that, for $1\le i\le r$, the $\G_a$-action on $X_i$ has index $i$, and its canonical factorization is given by $\pi_0\pi_1\cdots\pi_{i-1}$.


\section{$\G_a$-Actions on $\A^3$}\label{dim-3}

Suppose that $\rho :\G_a\times\A^3\to \A^3$
is a $\G_a$-action defined by the locally nilpotent derivation $D$ of $k^{[3]}$. Let $A=\krn D$ and $Y={\rm Spec}(A)$, and let $\pi :X\to Y$ be the quotient morphism for $\rho$. The following properties are known. 
\begin{itemize}
\item [(a)] $A\cong k^{[2]}$, or equivalently, $Y\cong \A^2$ (due to Miyanishi).
\item [(b)] $\pi$ is surjective (due to Bonnett).
\item [(c)] The plinth ideal $I_1=A\cap DB$ is a principal ideal of $A$ (due to Daigle). Equivalently, there exists a local slice $r$ of $D$ for which $\mathcal{F}_1=A\oplus Ar$. 
\item [(d)] If $\rho$ is fixed-point free, then $\rho$ is a translation, i.e., given by $\rho (t,(x,y,z))=(x,y,z+t)$ for some coordinates $(x,y,z)$ on $\A^3$ (due to Kaliman). 
\end{itemize}
None of these properties generalizes to higher dimensional affine spaces. See \cite{Freudenburg.06}, Chap.\,5 for details about these results.

Suppose that the canonical factorization of $\pi$ is given as in line (\ref{quotient}) above. Let $C\subset Y$ be the curve defined by $I_1$, which is, in general, reducible. We have:
\begin{itemize}\setcounter{enumi}{4}
\item [(e)] Each irreducible component of $C$ is a polynomial curve (\cite{Kaliman.04}, Thm.\,5.2).
\item [(f)] $X_1\cong Y\times\A^1\cong\A^3$ and $\pi_0 :X_1\to Y$ is projection on the first factor.
\item [(g)] Given $p\in Y\setminus C$, the fiber $\pi_0^{-1}(p)\cong\A^1$ is a single orbit.
\item [(h)] $X_1^{\G_a}\subset \pi_0^{-1}(C)$
\item [(i)] The mapping $\pi_1\cdots\pi_{r-1}$ is a $\G_a$-equivariant birational endomorphism of $\A^3$
\end{itemize}
In particular, part (f) classifies the $\G_a$-actions on $\A^3$ of index one.

Some of the examples presented in the next section use the following fact. Let $L\subset H\subset X=\A^3$, where $H=\A^2$ is a coordinate plane and $L=\A^1$ is a coordinate line. 
Let $\beta :X^*\to X$ be the blow-up of $X$ along $L$, and let $H^*\subset X^*$ be the proper transform of $H$.
If $U\subset X^*$ is the open subset $U=X^*\setminus H^*$, then $U\cong\A^3$ and $U\hookrightarrow X^*\xrightarrow{\beta} X$ is a birational endomorphism of $\A^3$. 


\section{Examples}\label{examples}

\subsection{The $(1,2)$ Action on $\A^3$} 
Let $B_1=k[x,y,z]=k^{[3]}$ and define the derivation $D_1$ of $B$ by $D_1=x\frac{\partial}{\partial y}$. Let $X_1=\A^3$ and let $H\subset X_1$ be the plane defined by $x=0$. 
Then $X_1^{\G_a}=H$
for the $\G_a$-action on $X_1$ defined by $D_1$. The kernel of $D_1$ is $A=k[x,z]$, and if $Y={\rm Spec}(A)$, then the quotient map $\pi_0 :X_1\to Y$ is a standard projection of $\A^3$ onto $\A^2$. 

Let $C\subset H$ be the curve (a coordinate line) defined by $x=z+y^2=0$. 
Let $\beta :X_1^*\to X_1$ be the blow-up of $X_1$ along $C$, let $E\subset X_1^*$ be the exceptional divisor lying over $C$, and let $H^*\subset X_1^*$ be the strict transform of $H$.
If $X_2\subset X_1^*$ is the open subset $X_2=X_1^*\setminus H^*$, then (as observed above) $X_2\cong\A^3$ 
and the mapping 
\[
\pi_1 : X_2\hookrightarrow X_1^*\xrightarrow{\beta} X_1
\]
is a birational endomorphism of $\A^3$. 
Since $C\subset X_1^{\G_a}$, the $\G_a$-action on $X_1$ lifts to $X_2$ and $\pi_1$ is equivariant. 

The curve $C$ is defined by the ideal $I=xB_1+(z+y^2)B_1$, where $D_1I\subset I$. If $B_2=k[X_2]$, 
then $B_2=B_1[x^{-1}I]=k[x,y,u]$, where $u=\frac{z+y^2}{x}$. 
If $D_2$ is the extension of $D_1$ to $B_2$, then:
\[
 D_2x=0 \,\, ,\,\, D_2y=x \,\, ,\,\, D_2u=2y
 \]
Moreover, $A=\krn D_2=k[x,z]=k[x,xu-y^2]$. $D_2$ is the homogeneous $(1,2)$ derivation of $k^{[3]}$, which is of index 2;  see \cite{Freudenburg.06}, 5.1.5. 

The canonical factorization of the quotient morphism $\pi :X_2\to Y$ is given by: 
\[
X=X_2  \xrightarrow{\pi_1} X_1 \xrightarrow{\pi_0} X_0=Y
\]
Let $E_0\subset X_2$ be the plane defined by $x=0$, noting that $M=X_2^{\G_a}\subset E_0$ is the line defined by $x=y=0$. Let $L\subset Y$ be the line defined by $x=0$. 
Since $(B_2)_x=(B_1)_x=A_x[y]$, 
it follows that $\pi_1 :X_2\setminus E_0\to X_1\setminus H$ is an isomorphism, and $\pi_0^{-1}(q)$ is a single orbit (isomorphic to $\A^1$) for each $q\in Y\setminus L$. 

Consider the restriction:
\[
\pi : E_0\cong \A^2 \xrightarrow{\pi_1} H\cong\A^2\xrightarrow{\pi_0} L\cong\A^1
\]
Given $p\in L$, $\pi^{-1}(p)$ is a union $P_1\cup P_2$ of two lines in $E_0$ which are orbits in $X_2$ if $p\ne 0$, and $\pi^{-1}(0)=M$. 
The situation is depicted in {\it Fig.\,2}. 

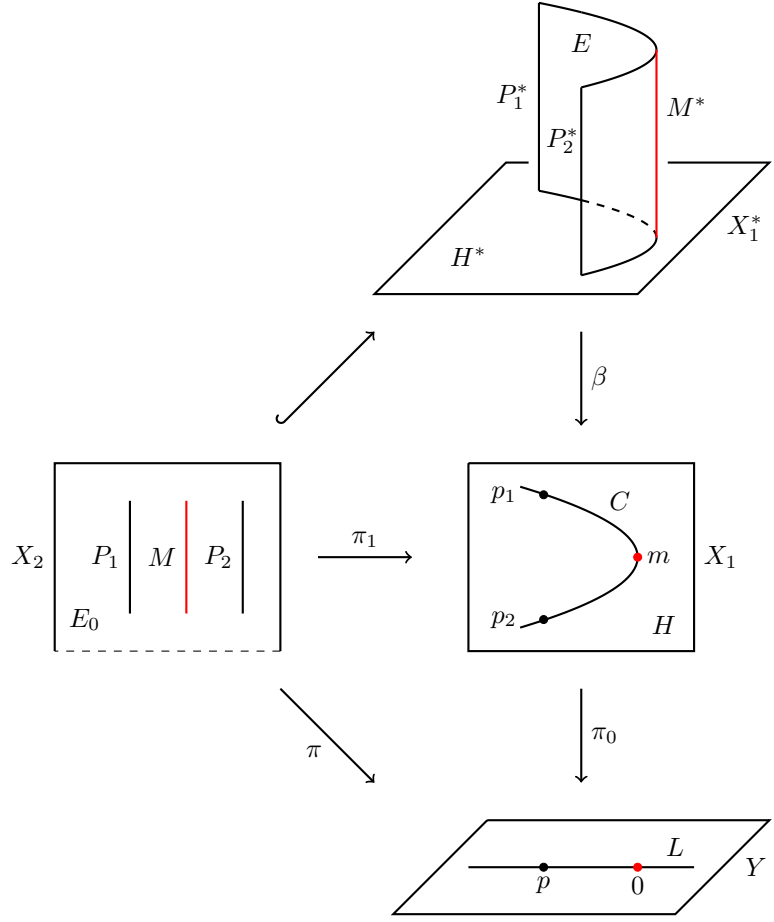
\begin{figure}

\begin{tikzpicture}

  \draw[domain=-1:0,thick,smooth,variable=\t]  plot ({-\t*\t},{.5*\t});
  \draw[domain=0:1,thick,smooth,variable=\t,dashed]  plot ({-\t*\t},{.5*\t});
   \draw[domain=1:1.25,thick,smooth,variable=\t]  plot ({-\t*\t},{.5*\t});

 \draw[domain=-1:0,thick,smooth,variable=\t]  plot ({-\t*\t},{.5*\t+2.5});
  \draw[domain=0:1,thick,smooth,variable=\t]  plot ({-\t*\t},{.5*\t+2.5});
   \draw[domain=1:1.25,thick,smooth,variable=\t]  plot ({-\t*\t},{.5*\t+2.5});
   
   \draw[thick] (-1,-.5) -- (-1,2);
 \draw[thick] (-1.5625,.625) -- node[left]{$P_1^*$} (-1.5625,3.125);
\draw[thick,red] (0,0) -- (0,2.5);
\node[above] at (.4,1.5){$M^*$};

  \node[above] at (-1,2.35){$E$};
   \node[above] at (-1.25,1){$P_2^*$};

\draw[thick] (.15,1) -- (1.5,1) -- node[right]{$\,\,\,X_1^*$}(-.25,-.75) -- (-3.75,-.75) -- (-2,1) -- (-1.7,1);
 \node[above] at (-2.5,-.5){$H^*$};
 
  \draw[->,thick] (-1,-1.25) -- node[right]{$\beta$} (-1,-2.5);

\draw[thick] (-2.5,-3) -- (.5,-3) -- node[right]{$X_1$} (.5,-5.5) -- (-2.5,-5.5) -- (-2.5,-3);
\draw[domain=-1.25:1.25,thick,smooth,variable=\t]  plot ({-\t*\t-.25},{.75*\t - 4.25});
 \filldraw[red] (-.25,-4.25) circle (1.5pt) node[right,black]{$m$};
 \filldraw (-1.5,-5.08) circle (1.5pt) node[left]{$p_2\,\,\,\,$};
 \filldraw (-1.5,-3.42) circle (1.5pt) node[left]{$p_1\,\,\,\,$};
 \node[right] at (-.75,-3.5){$C$};
 \node[above]  at (.1,-5.4){$H$};

 \draw[->,thick] (-1,-6) -- node[right]{$\pi_0$} (-1,-7.25);
 
 \draw[thick] (-2.25,-7.75) -- (1.5,-7.75) -- node[right]{$\,\,\,Y$}(.25,-9) -- (-3.5,-9) -- (-2.25,-7.75);
 \draw[thick] (.5,-8.375) -- (-2.5,-8.375);
 \filldraw (-1.5,-8.375) circle (1.5pt) node[below]{$p$};
  \filldraw[red] (-.25,-8.375) circle (1.5pt) node[below,black]{$0$};
  \node[above] at (.25,-8.35){$L$};
 
 \draw[thick] (-8,-5.5) -- node[left]{$X_2$}(-8,-3) -- (-5,-3) -- (-5,-5.5);
 \draw[dashed] (-5,-5.5) -- (-8,-5.5);
 \draw[->,thick] (-4.5,-4.25) -- node[above]{$\pi_1$}(-3.25,-4.25);
 \draw[right hook->,thick] (-5,-2.5) -- (-3.75,-1.25);
 
 \draw[thick] (-7,-5) -- node[left]{$P_1$}(-7,-3.5);
 \draw[thick,red] (-6.25,-5) -- node[left,black]{$M$}(-6.25,-3.5);
 \draw[thick] (-5.5,-5) -- node[left]{$P_2$}(-5.5,-3.5);
  \draw[->,thick] (-5,-6) -- node[below]{$\pi\quad$}(-3.75,-7.25);
  
  \node[above] at (-7.6,-5.35){$E_0$};
  
\end{tikzpicture}
\caption{Canonical Factorization for the $(1,2)$ $\G_a$-Action on $\A^3$}

\end{figure}


\subsection{The $(2,5)$ Action on $\A^3$}\label{(2,5)} This example is considerably more complicated than the preceding example. It is of rank three, meaning that the ring of invariants for the $\G_a$-action on $\A^3$ does not contain a variable.  

\subsubsection{The $(2,5)$ Derivation} The standard $\Z$-grading of the polynomial ring $B=k[x,y,z]=k^{[3]}$ is defined by letting $x,y,z$ be homogeneous of degree one.  
Define homogeneous elements $F,G,R,S\in B$ by:
 \[
 F=xz-y^2 \,\, ,\quad G= zF^2+2x^2yF+x^5  \,\, ,\quad R=x^3+yF \,\, ,\quad S=x^2y+zF 
\]
Observe the following relations. 
\[
F^3+R^2=xG\,\, ,\quad x^2R+FS=G\,\, ,\quad xS-yR=F^2
\]
The homogeneous $k$-derivation $D$ of $B$ defined by the jacobian determinant
\[
Dh=\frac{\partial (F,G,h)}{\partial (x,y,z)} \quad (h\in B)
\]
is locally nilpotent, and if $A=\krn D$, then $A=k[F,G]$; see \cite{Freudenburg.06}, 5.4. 
$D$ is the homogeneous (2,5) derivation of $B$ and the corresponding $\G_a$-action is the (2,5) $\G_a$-action on $\A^3$.
Note the following images. 
\[
DR=-FG \,\, ,\quad Dx= -2FR \,\, ,\quad DS = x(5xG-4F^3)  \,\, ,\quad Dy = 6x^2R-G \,\, ,\quad Dz = 2x(5yR+F^2)
\] 
In particular, $\deg_DR=1$, $\deg_Dx=2$, $\deg_DS=5$, $\deg_Dy=6$ and $\deg_Dz=10$. 

\subsubsection{The Module $\mathcal{F}_{10}$} 
Define the submodule $N_0$ of $\mathcal{F}_{10}$ by:
\[
N_0=\mathcal{G}_{10}(R)=A+AR+AR^2+AR^3+AR^4+AR^5+AR^6+AR^7+AR^8+AR^9+AR^{10}
\]
Since $R^2=xG-F^3$, we see that:
\[
N_0\subset N_1:=A+AR+Ax+AxR+Ax^2+Ax^2R+Ax^3+Ax^3R+Ax^4+Ax^4R +Ax^5   \subset \mathcal{F}_{10}
\]
Since $x^3=R-yF$ we see that:
\[
N_1\subset N_2:=A+AR+Ax+AxR+Ax^2+Ax^2R+Ay+Ax^3R+Axy+Ax^4R+Ax^2y      \subset \mathcal{F}_{10}
\]
Since $x^2R=G-FS$ we see that:
\[
N_2\subset N_3:=A+AR+Ax+AxR+Ax^2+AS+Ay+AxS+Axy+Ax^2S +Ax^2y     \subset \mathcal{F}_{10}
\]
Since $x^2y=S-zF$ we see that:
\[
N_3\subset M:=A+AR+Ax+AxR+Ax^2+AS+Ay+AxS+Axy+Ax^2S +Az\subset \mathcal{F}_{10}
\]
Note that $F^2M\subset N_1\subset M$.
 \begin{lemma}\label{lemma1} $FB\cap M=FM$
 \end{lemma}
 
 \begin{proof} Let $\pi_F:B\to B/FB$ be the canonical surjection. Let $\bar{b}=\pi_F(b)$ for $b\in B$, $\bar{A}=\pi_F(A)$, $\bar{M}=\pi_F(M)$ and $\bar{B}=\pi_F(B)$. 
 Since $F=xz-y^2$, we see that:
 \begin{equation}\label{B/FB}
 k[\bar{x},\bar{z}]=k^{[2]} \quad {\rm and}\quad \bar{B}=k[\bar{x},\bar{z}]\oplus k[\bar{x},\bar{z}]\bar{y}
 \end{equation}
 We have:
 \[
 \bar{G}=\bar{x}^5\,\, ,\,\, \bar{R}=\bar{x}^3\,\, ,\,\, \bar{S}=\bar{x}^2\bar{y}\,\, ,\,\, \bar{A}=k[\bar{x}^5]
 \]
 Define:
 \begin{equation}\label{big-O}
 \mathcal{O}=k[\bar{x}]=\bar{A}\oplus\bar{A}\bar{x}+\oplus\bar{A}\bar{x}^2\oplus\bar{A}\bar{x}^3\oplus\bar{A}\bar{x}^4 
  \end{equation}
 Then:
 \begin{eqnarray*}
 \bar{M} &=& \bar{A}  +\bar{A}\bar{x}^3+  \bar{A} \bar{x}+  \bar{A}\bar{x}^4+  \bar{A}\bar{x}^2 + 
 \bar{A}\bar{x}^2\bar{y}+  \bar{A}\bar{y} + \bar{A}\bar{x}^3\bar{y}+\bar{A}\bar{x}\bar{y} +\bar{A}\bar{x}^4\bar{y}  +\bar{A}\bar{z}
 \\
 &=& \mathcal{O}+\mathcal{O}\bar{y} +\bar{A}\bar{z}
 \end{eqnarray*}
From (\ref{B/FB}) and (\ref{big-O}) it follows that $\bar{M}$ is a free $\bar{A}$-module of rank 11. Since $FB\cap A=FA$ by (\ref{krn-ideal}), we conclude that $FB\cap M=FM$.
 \end{proof}
 
 \begin{lemma}\label{lemma2} $GB\cap M=GM$
 \end{lemma}
 
 \begin{proof} Let $\pi_G:B\to B/GB$ be the canonical surjection and let $\bar{D}$ be the locally nilpotent derivation of $B/GB$ induced by $D$. Let $\bar{b}=\pi_G(b)$ for $b\in B$, $\bar{A}=\pi_G(A)$ and $\bar{N}_1=\pi_G(N_1)$. Then $\bar{A}=k[\bar{F}]$. Define 
 $\mathcal{R}=k[\bar{F},\bar{R}]$. Since $xG=F^3+R^2$, we have $\bar{F}^3+\bar{R}^2=0$. In addition, note that $\mathcal{R}\subset\krn\bar{D}$ and $\bar{D}\bar{x}\ne 0$. We thus have:
 \[
 \mathcal{R}=\bar{A}\oplus \bar{A}\bar{R} \quad {\rm and}\quad \mathcal{R}[\bar{x}]=\mathcal{R}^{[1]}
 \]
 Therefore,
 \begin{eqnarray*}
 \bar{N}_1 &=& \bar{A} + \bar{A}\bar{R}+  \bar{A} \bar{x}+  \bar{A}\bar{x}\bar{R}+  \bar{A}\bar{x}^2+  \bar{A}\bar{x}^2\bar{R}+  \bar{A}\bar{x}^3 +\bar{A}\bar{x}^3\bar{R} 
 +\bar{A}\bar{x}^4+\bar{A}\bar{x}^4\bar{R}+\bar{A}\bar{x}^5 \\
 &=& \mathcal{R}\oplus\mathcal{R}\bar{x}\oplus\mathcal{R}\bar{x}^2\oplus\mathcal{R}\bar{x}^3\oplus\mathcal{R}\bar{x}^4  \oplus \bar{A}\bar{x}^5
 \end{eqnarray*}
 is a free $\bar{A}$-module of rank 11. Since $GB\cap A=GA$ by (\ref{krn-ideal}), we conclude that:
 \begin{equation}\label{N1}
  GB\cap N_1=GN_1
  \end{equation}
 
 Suppose that $Gw\in M$ for some $w\in B$. Then $F^2Gw\in F^2M\subset N_1$, so $F^2Gw\in GB\cap N_1=GN_1$ by (\ref{N1}). Therefore, $F^2w\in N_1\subset M$,
 so $F^2w\in F^2B\cap M=F^2M$ by {\it Lemma\,\ref{lemma1}}, and $w\in M$. 
\end{proof}

\begin{theorem} $M=\mathcal{F}_{10}$. 
\end{theorem}

\begin{proof} By {\it Thm.\,\ref{main} (c)}, it will suffice to show $FG\cdot B\cap M=FG\cdot M$. This follows immediately from {\it Lemma\,\ref{lemma1}} and {\it Lemma\,\ref{lemma2}}.
\end{proof}

Note that, by degree considerations, $\mathcal{F}_{10}$ is a free $A$-module of rank 11. Therefore:
\begin{eqnarray*}
\mathcal{F}_0 &=& A \\
\mathcal{F}_1&=& A+AR \\
\mathcal{F}_2&=& A+AR+Ax \\
\mathcal{F}_3&=& A+AR+Ax+AxR \\
\mathcal{F}_4&=& A+AR +Ax+AxR+Ax^2\\
\mathcal{F}_5&=& A+AR +Ax+AxR+Ax^2+AS\\
\mathcal{F}_6&=& A+AR +Ax+AxR+Ax^2+AS+Ay\\
\mathcal{F}_7&=& A+AR +Ax+AxR+Ax^2+AS+Ay+AxS\\
\mathcal{F}_8&=& A+AR +Ax+AxR+Ax^2+AS+Ay+AxS+Axy\\
\mathcal{F}_9&=& A+AR +Ax+AxR+Ax^2+AS+Ay+AxS+Axy+Ax^2S\\
\mathcal{F}_{10}&=& A+AR +Ax+AxR+Ax^2+AS+Ay+AxS+Axy+Ax^2S+Az
\end{eqnarray*}

\subsubsection{Degree Resolution}

Results above give the degree resolution of $B$ induced by $D$.
\medskip

\begin{enumerate}
\item $B_0=\mathcal{F}_0=A=k[F,G]=k^{[2]}$
\smallskip
\item $B_1=k[\mathcal{F}_1]=A[R]=k[F,G,R]=k^{[3]}$  
\smallskip
\item $B_2=k[\mathcal{F}_2]=B_1[x]=k[F,G,R,x]$ where $xG=F^3+R^2$
\smallskip
\item $B_4=B_3=B_2$
\smallskip
\item $B_5=k[\mathcal{F}_5]=B_2[S]=k[F, R, x, S]$ where $F(xS-F^2)=R(R-x^3)$
\smallskip
\item $B_6=k[\mathcal{F}_6]=B_5[y]=k[F,x,S,y]$ where $x(S-x^2y)=F(F+y^2)$
\smallskip
\item $B_9=B_8=B_7=B_6$
\smallskip
\item $B_{10}=k[\mathcal{F}_{10}]=B_6[z]=B$
\end{enumerate}
\medskip

It follows that $\mathcal{N}_A(B)=\{ n_0,...,n_5\}=\{ 0,1,2,5,6,10\}$ and ${\rm index}_B(A)=5$. Observe that $B_0$, $B_1$, $B_2$, $B_{10}$ are UFDs, whereas neither $B_5$ nor $B_6$ is a UFD. 

\subsubsection{Fixed Points} Let $X_i={\rm Spec}(B_{n_i})$ for $i=0,...,5$.
\begin{enumerate}
\item $X_1^{\G_a}=\mathcal{V}(FG)\subset X_1$, which defines two planes in $\A^3$.
\smallskip
\item $X_2^{\G_a}=\mathcal{V}(F)\subset X_2$, which defines a cone. 
\smallskip
\item $X_3^{\G_a}=\mathcal{V}(F,R)\subset X_3$, which defines a plane. 
\smallskip
\item $X_4^{\G_a}=\mathcal{V}(F,x)\subset X_4$, which defines a plane. 
\smallskip
\item $X_5^{\G_a}=\mathcal{V}(x,y)\subset X_5$, which defines a line in $\A^3$. 
\end{enumerate} 

\subsubsection{Affine Modifications}
 
We describe each ring $B_{n_{i+1}}$ as a $\G_a$-equivariant affine modification of $B_{n_i}$, $1\le i\le 4$. 

\begin{enumerate} 
\item $B_2=B_1[G^{-1}J_1]$ for $J_1=G\cdot B_1+(F^3+R^2)\cdot B_1$
\smallskip
\item $B_5=B_2[F^{-1}J_2]$ for $J_2=F\cdot B_2+(G-x^2R)\cdot B_2$
\smallskip
\item $B_6=B_5[F^{-1}J_5]$ for $J_5=F\cdot B_5+(R-x^3)\cdot B_5$
\smallskip
\item $B_{10}=B_6[F^{-1}J_6]$ for $J_6=F\cdot B_6+(S-x^2y)\cdot B_6$
\end{enumerate}

\subsubsection{Canonical Factorization}

Let $X={\rm Spec}(B)=\A^3$ and $Y={\rm Spec}(A)=\A^2$, and let $\pi :X\to Y$ be the quotient morphism. 
Over points $p\in Y$ defined by $F=\alpha$, $G=\beta$, the fiber $\pi^{-1}(p)$ is a line which is a single orbit if $\alpha ,\beta\ne 0$; a union of five lines which are orbits if $\alpha =0$, 
$\beta\ne 0$; a union of two lines which are orbits if $\alpha\ne 0$, $\beta =0$; and a line of fixed points if $\alpha =\beta =0$.

Let $\pi_i:X_{i+1}\to X_i$ be the morphism induced by the inclusion $B_{n_i}\to B_{n_{i+1}}$, $0\le i\le 4$. The canonical factorization of $\pi$ is given by:
\[
X=X_5\xrightarrow{\pi_4}  X_4\xrightarrow{\pi_3}  X_3\xrightarrow{\pi_2}  X_2\xrightarrow{\pi_1}  X_1\xrightarrow{\pi_0}  X_0=Y
\]
We consider each mapping $\pi_i$ individually.

\begin{enumerate}

\item For $\pi_0$, we have $X_0=\A^2$, $X_1=X_0\times\A^1$ and $\pi_0$ is projection on the first factor. 
\medskip

\item For $\pi_1$, let $W_1=\mathcal{V}(G)\subset X_1$ and $W_2=\mathcal{V}(G)\subset X_2$. Since $(B_1)_G=(B_2)_G$, the mapping
\[
\pi_1: X_2\setminus W_2\to X_1\setminus W_1
\]
is an isomorphism. 
 We find that $W_1=\A^2$ and $\pi_1(W_2)=C$, where $C$ is the cuspidal cubic curve $\mathcal{V}(G,F^3+R^2)$ in $W_1$, and that $W_2=C\times\A^1$, 
where the restriction of $\pi_1$ to $W_2$ is projection on the first factor. The image of $\pi_1$ excludes $W_1\setminus C$. 
\medskip

\item For $\pi_2$, let $V_2=X_2^{\G_a}=\mathcal{V}(F)\subset X_2$ and $W_3=\mathcal{V}(F)\subset X_3$. Since $(B_2)_F=(B_5)_F$, the mapping
\[
\pi_2: X_3\setminus W_3\to X_2\setminus V_2
\]
is an isomorphism. 
We find that $\pi_2(W_3)=Z$, where $Z$ is the union of two lines $\mathcal{V}(F,G,R)$ and $\mathcal{V}(F,G-x^5,R-x^3)$ in $V_2$,
and that $W_3=Z\times\A^1$, where the restriction of $\pi_2$ to $W_3$ is projection on the first factor. 
The image of $\pi_2$ excludes $W_2\setminus Z$. 
\medskip

\item For $\pi_3$, let $V_3=X_3^{\G_a}=\mathcal{V}(F,R)\subset X_3$ and $V_4=X_4^{\G_a}=\mathcal{V}(F,x)=\mathcal{V}(F,R)\subset X_4$.
Since $(B_5)_F=(B_6)_F$ and $(B_5)_R=(B_6)_R$, the mapping
\[
\pi_3: X_4\setminus V_4\to X_3\setminus V_3
\]
is an isomorphism.
We find that $\pi_3(V_4)=L$, where $L$ is the line $\mathcal{V}(F,R,x)\subset X_3$, 
and that $V_4=L\times\A^1$ (a plane), where the restriction of $\pi_3$ to $V_4$ is projection on the first factor. 
The image of $\pi_3$ excludes $V_3\setminus L$. 
\medskip

\item For $\pi_4$, let $X_5^{\G_a}=\mathcal{V}(F,x)=\mathcal{V}(x,y)\subset X_5$.
Since $(B_6)_F=B_F$ and $(B_6)_x=B_x$, the mapping 
\[
\pi_4: X_5\setminus V_5\to X_4\setminus V_4
\]
is an isomorphism. 
We find that $\pi_4(V_5)=P$, where $P$ is the point $\mathcal{V}(F,x,S,y)$, and that $V_5=P\times\A^1$ (a line), where the restriction of $\pi_4$ to $V_5$ is projection on the first factor. 
The image of $\pi_4$ excludes $V_4\setminus P$. 
\end{enumerate}
\medskip

The affine modification $\pi_4 :X_5\to X_4$ differs from the first three in that its exceptional locus is one-dimensional. One way to understand this situation is to view 
$X_4$ as a subvariety of $\A^4$ given by $xT=F(F+y^2)$ in coordinates $x,F,y,T$ (where $T=S-x^2y$). According to {\it Prop.\,2.1} of \cite{Kaliman.Zaidenberg.99}, we can view $\pi_4$ as the restriction of the affine modification of $\A^4$ along the divisor $\{ F=0\}$ with center $\{ F=T=0\}$. If $\beta :\mathcal{X}\to\A^4$ is the associated morphism, then 
$\mathcal{X}\cong\A^4$ with coordinates $x,F,y,\frac{T}{F}$; and $X_5\subset\mathcal{X}$ is the hyperplane defined by $x\frac{T}{F}=F+y^2$.

\subsubsection{A $D$-Basis for $B$}

Let $p:B\to B/zB=k[x,y]\cong k^{[2]}$ be the standard surjection, and set $\bar{A}=p(A)$ and $\bar{\mathcal{F}}_n=p(\mathcal{F}_n)$. 
\begin{proposition} $\bar{\mathcal{F}}_9=k[x,y]$ and $\bar{\mathcal{F}}_9$ is a free $\bar{A}$-module of rank 10.
\end{proposition}

\begin{proof} We have $\bar{A}=k[y^2,x^5+2x^2y^3]$.
We see that $\bar{A}\subset \bar{A}[y]$ is an integral extension of degree 2, and that $\bar{A}[y]\subset k[x,y]$ is an integral extension of degree 5. Therefore, $k[x,y]$ is a free $\bar{A}$-module of rank 10 with basis:
\[
\mathcal{B}=\{ 1,x,x^2,x^3,x^4,y,xy,x^2y,x^3y,x^4y\}
\]
Observe that $\mathcal{F}_9$ contains the set
\[
\mathcal{C}=\{ 1,x,x^2,R-yF, x(R-yF), y,xy,S,xS,x^2S\}
\]
and that $p(\mathcal{C})=\mathcal{B}$. Therefore, $\bar{\mathcal{F}}_9=k[x,y]$.
\end{proof}

\begin{corollary}\label{dim3-free} $B=\bigoplus_{i\ge 0}\mathcal{F}_9\cdot z^i$
\end{corollary}

\begin{proof} Set $N=\sum_{i\ge 0}\mathcal{F}_9\cdot z^i$. By {\it Lemma\,\ref{lemma2}}, we see that $N=\bigoplus_{i\ge 0}\mathcal{F}_9\cdot z^i$.

Consider the descending chain of submodules:
\[
B\supset N+zB\supset N+z^2B\supset\cdots
\]
By the proposition, we see that $B=N+zB$. Since $N+zN=N$, it follows that $N+z^nB=B$ for every $n\ge 0$. 
Given nonzero $f\in B$, choose an integer $n>\deg_D f$ and write $f=\sum_{0\le i\le n-1}a_iz^i+z^nb$, $a_i\in\mathcal{F}_9$ and $b\in B$. 
If $b\ne 0$, then $\deg_D(a_iz^i)<\deg_D(z^nb)$ for $0\le i\le n-1$, which implies
\[
\deg_Df=\deg_D(z^nb)=10n+\deg b
\]
a contradiction. Therefore, $b=0$ and $f\in N$.  
\end{proof}

It is shown above that $\mathcal{F}_9$ is a free $A$-module which admits a $D$-basis. Thus, the equality in {\it Cor.\,\ref{dim3-free}}, together with {\it Lemma\,\ref{lemma2}(b)} and the $D$-basis of 
$\mathcal{F}_9$, give a $D$-basis for each $\mathcal{F}_n$, $n\ge 0$.  It follows that $B$ is a free $A$-module which admits a $D$-basis.  

\subsubsection{Associated Graded Ring}  The foregoing calculations show the following.
\[
{\rm Gr}_D(B) = A[FGt, F^2Gt^2, F^4G^3t^5, F^5G^3t^6, F^8G^5t^{10}] \subset A[t]=A^{[1]}
\]
In particular, ${\rm Gr}_D(B)$ is finitely generated as a $k$-algebra.

\subsection{Russell Cubic Threefold} (See also \cite{Kaliman.Zaidenberg.99}, Examples 1.5, 3.2.) Let $\C [x,y,z,t]=\C^{[4]}$. The Russell cubic threefold $X\subset\C^4$ is defined by the zero set of the polynomial 
$x+x^2y+z^2+t^3$. $X$ is smooth, contractible and factorial, and $X$ is diffeomorphic to $\R^8$. On the other hand, the function $x$ restricted to $X$ is an invariant of every $\G_a$-action on $X$, which implies that $X$ is not isomorphic to $\C^3$ as a complex algebraic variety; see, for example, \cite{Kaliman.09}. $X$ is an example of an exotic affine space. 

Let the coordinate ring $B=\C [X]$ be given by $B=\C [x,y,z,t]$, where $x+x^2y+z^2+t^3=0$. 
The derivation $D=x^2\frac{\partial}{\partial z}-2z\frac{\partial}{\partial y}$ of $B$ is locally nilpotent, and if $A=\krn D$, then $A=\C [x,t]$. 

\begin{lemma} ${\rm pl}(D)=x^2A$
\end{lemma}

\begin{proof} Let $\mathcal{F}_1$ be the first degree module for $D$, noting that $A+Az\subset\mathcal{F}_1$. 

Let $\bar{B}=B/xB$ and let $\bar{D}$ be the locally nilpotent derivation on $\bar{B}$ induced by $D$. Then:
\[
\bar{B} = \frac{\C [z,t]}{(z^2+t^3)}[y] \quad {\rm and}\quad \krn \bar{D} =  \frac{\C [z,t]}{(z^2+t^3)} = \C [t] \oplus \C [t]\cdot z
\]
If $g\in\mathcal{F}_1$ and $xg\in A+Az$, then $Dg\in xA$. Write $Dg=xP(t) +x^2h$ for some $P\in\C [t]$ and $h\in A$.\footnote{$Dg$ has no constant term, since 0 is a fixed point of the $\G_a$-action.}
Then $g-zh\in A+Az$ and $D(g-zh)=xP(t)$. 
Since $g-zh, z\in\mathcal{F}_1$, we have:
\[
(g-zh)x^2-xP(t)z\in A \quad\Rightarrow\quad (g-zh)x-P(t)z\in A 
\]
Modulo $xB$, this implies:
\[
 -P(t)z\in \bar{A}=\C [t]\subset\krn\bar{D}=\C [t]\oplus \C [t]\cdot z
\quad\Rightarrow\quad P(t)\equiv 0 \quad\Rightarrow\quad P(t) \in xB
\]
Since $\C [x,t]\cong\C^{[2]}$, it follows that $P(t)=0$ and $D(g-zh)=0$. Therefore, $g\in A+Az$. By {\it Thm.\,\ref{main}}, $\mathcal{F}_1=A+Az$, which implies
${\rm pl}(D)=D\mathcal{F}_1=x^2A$. 
\end{proof}

By this lemma, $B_1=A[z]=\C [x,z,t]\cong\C^{[3]}$. In addition, since $\deg_D(y)=2$, $B_2=B_1[y]=B$. We find that $B_2=B_1[x^{-2}I]$, where 
$I=x^2B_1+(x+z^2+t^3)B_1$. 

Let $X_1={\rm Spec}(B_1)\cong\C^3$, and let $Y={\rm Spec}(A)$. The canonical resolution of the quotient morphism $\pi :X\to Y$ is given by:
\[
X=X_2  \xrightarrow{\pi_1} X_1\xrightarrow{\pi_0} X_0=Y
\]
We have that $X_1=Y\times\C^1$, and $\pi_0$ is projection on the first factor. 

Let $W_1=\mathcal{V}(x)\subset X_1$ and $W=\mathcal{V}(x)\subset X$. Then $W_1\cong\C^2$ and $W\cong C\times\C^1$, where $C$ is the cuspidal cubic curve $\mathcal{V}(x,z^2+t^3)$. Since $(B_1)_x=B_x$, $\pi_1 :X\setminus W\to X_1\setminus W_1$ is an isomorphism. Otherwise, $\pi_1(W)=C$. 


\subsection{Winkelmann's Example} In \cite{Winkelmann.90}, Winkelmann gave the first examples of free $\G_a$-actions on affine space which are not translations. We analyze the smallest of these examples, which is in dimension four. 

Let $B=k[x,y,z,u]=k^{[4]}$, and define $F\in B$ by $F=2xz-y^2$. Define $D\in {\rm LND}(B)$ by:
\[
Dx=0\,\, ,\,\, Dy=x\,\, ,\,\, Dz=y\,\, ,\,\, Du=F+1
\]
Let $A=\krn D$, and let $\pi :X\to Y$ be the induced quotient morphism, where $X={\rm Spec}(B)$ and $Y={\rm Spec}(A)$.
Since $xB+yB+(F+1)B=B$, the induced $\G_a$-action on $\A^4$ is fixed-point free.

Let $\mathcal{F}_n$ be the $A$-module $\mathcal{F}_n=\krn D^{n+1}$ and let $B_n=k[\mathcal{F}_n]$, $n\ge 0$. 
Since $x,y,z,u\in\mathcal{F}_2$, we see that $B=B_2$, so the degree resolution of $A$ is given by:
\[
A=B_0\subset B_1\subset B_2=B
\]
Hence, if $X_i={\rm Spec}(B_i)$, then the canonical factorization of $\pi$ is given by:
\[
X=\A^4=X_2  \xrightarrow{\pi_1} X_1\xrightarrow{\pi_0} X_0=Y
\]

\subsubsection{The Plinth Ideal}
Observe the following. 
\begin{enumerate}
\item Since $y$ and $u$ are local slices, $G:=uDy-yDu\in A$. 
\item If $T=yu-2z(F+1)$, then $DT=G$. 
\item Since $u$ and $T$ are local slices, $H:=uDT-TDu\in A$.
\item It is easy to show that $A=k[x,F,G,H]$. The prime relation for this ring is: 
\[
xH=G^2+F(F+1)
\]
\end{enumerate}

\begin{lemma}\label{Wink-lemma} ${\rm pl}(D)=(x,F+1,G)$. 
\end{lemma}

\begin{proof}
Let $\sigma :B\to B/xB$ be the standard surjection, let $\bar{A}=\sigma (A)$ and let $\bar{H}=\sigma (H)$. Then:
\[
\bar{A}=k[y^2,y(y^2-1),\bar{H}] = k[y^2,y(y^2-1)]^{[1]}
\]
Let $J\subset A$ be the ideal $J=(x,F+1,G)$, and let $f\in I_1$ be given. Since $J\subset I_1$ and $A/J=k[H]$, it suffices to assume $f\in k[H]$. 

Write $f=P(H)$ for $P\in k^{[1]}$ and let $L\in B$ be such that $DL=f$. Then:
\[
yDL-LDy = yP(H)-xL\in A \quad \Rightarrow\quad  yP(\bar{H})\in\bar{A}
\]
If $P(\bar{H})\ne 0$, choose $\lambda\in k$ so that $P(\lambda )\ne 0$. Then:
\[
yP(\lambda )\in k[y^2, y(y^2-1)] \quad \Rightarrow\quad  k[y^2, y(y^2-1)] = k[y]
\]
But this is clearly a contradiction. Therefore, $P(\bar{H})=0$, which implies $f=P(H)=0$. Therefore, $J=I_1$. 
\end{proof}

\subsubsection{The mapping $\pi_0$}

Let $W_0\subset Y$ be defined by $W_0={\rm Spec}(A/I_1)$. {\it Lemma\,\ref{Wink-lemma}} implies that $A/I_1 = k[H]\cong k^{[1]}$, so $W_0\cong\A^1$. Let $V=Y\setminus W_0$. Then 
there is an open set $U\subset X_1$ such that $U=V\times\A^1$ and $\pi_0$ is projection on the first factor. 
In particular, let $D_1$ be the restriction of $D$ to $B_1$. Since $y,u,T$ are local slices, we have:
\[
(B_1)_x=A_x[y] \,\, ,\quad (B_1)_{F+1}=A_{F+1}[u]\,\, ,\quad (B_1)_G=A_G[T]
\]
Therefore, $U=U_x\cup U_{F+1}\cup U_T$, where $U_x=\{ x\ne 0\}$, $U_{F+1}=\{F+1\ne 0\}$ and $U_T=\{ T\ne 0\}$.

{\it Lemma\,\ref{Wink-lemma}} further implies that $\mathcal{F}_1=A+Ay+Au+AT$. Relations in this module are given by:
\[
yG-xT=F(F+1)\,\, ,\,\, xu-y(F+1)=G\,\, ,\,\, uG-(F+1)T=H
\]
The ring $B_1$ is given by $B_1=k[x,y,u,F,N]$, where $N=z(F+1)$. The prime relation in this ring is $2xN=(F+y^2)(F+1)$. 

The closed set $W_1=X_1\setminus U$ is the set of fixed points $X_1^{\G_a}$, defined by the ideal:
\[
 I_1B_1=xB_1+(F+1)B_1
 \]
 We find that $W_1\cong\A^3$. Therefore, 
$\pi_0$ restricts to $\pi_0 : W_1\cong\A^3\to W_0\cong\A^1$ and we find that $\pi_0(W_1)=p$, where $p\in Y$ is the point defined by the ideal $(x,F+1,G,H)\subset A$. Therefore, 
$\pi_0^{-1}(W_0\setminus\{ p\})=\emptyset$. 

\subsubsection{The mapping $\pi_1$} Let $W\subset X=\A^4$ be defined by the ideal $xB+(F+1)B$. Then $W=W_+\cup W_-$, where $W_+=\{ x=0,y=1\}\cong\A^2$ and 
$W_-=\{ x=0,y=-1\}\cong\A^2$.
Since $(B_1)_x=B_x$ and $(B_1)_{F+1}=B_{F+1}$, we see that $\pi_1$ maps $X\setminus W$ isomorphically to $X_1\setminus W_1$. 

Using coordinate functions $(y,u,N)$ on $W_1=\A^3$, we find that $\pi_1(W_+)=L_+\cong\A^1$ and $\pi_1(W_-)=L_-\cong\A^1$, where $L_+=\{ y=1,N=0\}$ and $L_-=\{ y=-1,N=0\}$. 

\subsubsection{Summary} The canonical factorization of $\pi$ splits as follows:
\[
X\setminus W\underset{\sim}{\xrightarrow{\pi_1}} X_1\setminus X_1^{\G_a}\cong V\times\A^1 \xrightarrow{\pi_0} V
\]
and:
\[
W\cong\A^2\cup\A^2\xrightarrow{\pi_1} X_1^{\G_a}\cong\A^3\xrightarrow{\pi_0}W_0\cong \A^1\,\, ,\quad \pi_1(W)=\A^1\cup\A^1 \,\, {\rm and}\,\, \pi_0(X_1^{\G_a}) = p
\]
As an affine modification, we find that $B=B_1[(F+1)^{-1}J]$, where $J=(F+1)B_1 + NB_1$. 


\section{A Triangular $R$-Derivation of $R^{[3]}$ With a Slice}

\subsection{The Derivation $\delta$ of $\C^{[4]}$} 
Let $R=\C [x,y]=\C^{[2]}$ and $B=R[z,u]=R^{[2]}$, and define $p,v\in B$ by $p=yu+z^2$ and $v=xz+yp$. 
Define the triangular $R$-derivation of $B$ by:
\[
\delta = v_u\partial_z-v_z\partial_u = y^2\partial_z-(x+2yz)\partial_u
\]
If $A=\krn\delta$, then $A=R[v]$. Let $\{\mathcal{F}_n\}_{n\ge 0}$ be the degree modules for $\delta$. Since $v, z\in\mathcal{F}_1$, we see that $p\in\mathcal{F}_1$. Modulo $y$, we have $p\equiv z^2$ and $v\equiv xz$, meaning that $xp-vz=yq$ for some $q\in\mathcal{F}_1$. We find that $q=xu-zp$. 

\begin{proposition}\label{plinth} For the derivation $\delta$:
\begin{itemize}
\item [{\bf (a)}] $\mathcal{F}_1=A+Az+Ap+Aq$
\item [{\bf (b)}] A complete set of $A$-relations for $\mathcal{F}_1$ is given by:
\[
xz+yp-v=0 \quad {\rm and} \quad vz-xp+yq=0
\]
\item [{\bf (c)}] ${\rm pl}(\delta )=y^2A+xyA+(x^2+yv)A$
\end{itemize}
\end{proposition}

\begin{proof} Define $M\subset\mathcal{F}_1$ by $M=A+Az+Ap+Aq$. Since $yp=v-xz$ and $yq=xp-vz$, we have:
\[
M = A + \C [x,v]p+Az+\C [x,v]q
\]
Suppose that $F\in\mathcal{F}_1$ and $xF\in M$. Write:
\[
xF=a_0+a_1p+a_2z+a_3q \quad {\rm where}\quad a_0,a_2\in A\,\, ,\,\, a_1,a_3\in\C [x,v]
\]
Modulo $x$, we have $0=\bar{a}_0+\bar{a}_1p+\bar{a}_2z-\bar{a}_3zp$ where $\bar{a}_0,\bar{a}_2\in\C [y,yp]$ and $\bar{a}_1,\bar{a}_3\in\C [yp]$. Since $\C [y,yp,z]\cong\C^{[3]}$, it follows that $\bar{a}_0+\bar{a}_1p=\bar{a}_2-\bar{a}_3p=0$. Since 
\[
\C [y,yp]+\C [yp]p = \C [y,yp]\oplus \C [yp]p
\]
we conclude that $\bar{a}_i=0$ for each $i$. If $a_i=xb_i$ for $b_i\in B$, then $b_i\in A$, since $A$ is factorially closed. Therefore, $F\in M$. By {\it Thm.\,\ref{main}}, it follows that 
$M=\mathcal{F}_1$. This proves part (a), and the same argument shows part (b). 

Since ${\rm pl}(\delta )=\delta\mathcal{F}_1$, part (c) follows by the observing that:
\[
\delta z=y^2\,\, ,\quad \delta p=-xy\,\, ,\quad \delta q=-(x^2+yv)
\]
\end{proof}

The {\bf V\'en\'ereau polynomial} $f\in A$ is defined by $f=y+xv$. It is well known that $\mathcal{V}(f)\cong\C^3$, 
but it is not known whether $f$ is a variable of $B$. See \cite{Freudenburg.06}. 

Observe that, if $r\in \mathcal{F}_1$ is defined by $r=yz+(v^4-3fv)p-xv^3q$, then $\delta r=f^3$. 

\begin{corollary} $f^2\not\in{\rm pl}(\delta )$
\end{corollary}

\begin{proof} 
Let $\C [X,Y,Z]=\C^{[3]}$ and $I=(Y^2,XY,X^2+YZ)$. It must be shown that $(Y+XZ)^2\not\in I$. 

Let $\C [X,Y,Z]=\bigoplus_{i\ge 0}V_i$ be the standard $\Z$-grading of $\C [X,Y,Z]$, where $V_i$ is the vector space of homogeneous forms of degree $i$. Then $I=\bigoplus_{i\ge 0}I_i$ is a graded ideal, where $I_i=I\cap V_i$. 

Let $\mathcal{W}=Y^2\cdot V_2+XY\cdot V_2$. Then:
\[
I_4=\mathcal{W}+(X^2+YZ)\cdot V_2 = \mathcal{W} + \C\cdot (X^2+YZ)Z^2
\]
Therefore, $\dim_{\C} I_4/\mathcal{W} = 1$. Note that the elements $X^2Z^2$ and $YZ^3$ are linearly independent modulo $\mathcal{W}$, since no element of $\mathcal{W}$ has a term supporting $X^2Z^2$ or $YZ^3$. Since $(X^2+YZ)Z^2\in I_4$, it follows that $X^2Z^2\not\in I_4$. 

If $(Y+XZ)^2\in I$, then $X^2Z^2\in I_4$, a contradiction. Therefore, $(Y+XZ)^2\not\in I$. 
\end{proof}

Extend $\delta$ to the triangular derivation $\Delta$ on $B[t]=B^{[1]}$ by $\Delta t=1+f+f^2$. If $s=(1-f)t+r$, then $\Delta s=1$. The kernel of $\Delta$ is the image of the Dixmier map induced by $s$, namely:
\[
\krn\Delta = R[\, z-y^2s\, ,\,\, u+(x+2yz)s-y^3s^2\, ,\,\, f^3t-(1+f+f^2)r\, ]
\]
 According to \cite{Freudenburg.09}, $\krn\Delta$ is a $\C^2$-fibration over $R$. The question is whether it is a trivial fibration.
 
 \begin{question} Do there exist $P,Q\in B[t]$ with $\krn\Delta = R[P,Q]$?
 \end{question}


\section{Conclusion}\label{last}

\subsection{Module Generators}  In {\it Thm.\,\ref{main}}, for the initial module $M_0\subset\mathcal{F}_n$, we can always take
$M_0=\mathcal{G}_n(r)$ for a local slice $r$. If $n\ge 2$ and $\mathcal{F}_1,...,\mathcal{F}_{n-1}$ are known, then a more efficient choice is:
\[
M_0=\sum_{1\le i\le n-1}\mathcal{F}_i\mathcal{F}_{n-i}
\]
We ask the following: If $B=B_N=k[\mathcal{F}_N]$, does it follow that, for each $n\ge 0$,
\[
\mathcal{F}_n=\sum_{e_1+2e_2+\cdots Ne_N=n}\mathcal{F}_1^{e_1}\cdots\mathcal{F}_N^{e_N} \quad ({\rm where}\,\, \mathcal{F}_i^0=A) \, {\rm ?}
\]
More generally, is ${\rm Gr}_D(B)$ finitely generated as a $k$-algebra (given that $B$ is affine)?

\subsection{Related Work} In his thesis, Alhajjar also gives an algorithm for finding degree modules $\mathcal{F}_n$ associated to a locally nilpotent derivation. His algorithm is very different than the one presented herein, employing what he terms a ``twisted embedding technique''. The algorithm we give is modeled after the algorithm of van den Essen for finding generators of the kernel of a locally nilpotent derivation, i.e., the ring of invariants of the corresponding $\G_a$-action. Van den Essen's algorithm, in turn, is a generalization of the technique used by Tan to find generators for the invariants of an irreducible representation of $\G_a$, a technique which was essentially already in use in the Nineteenth Century. 
See \cite{Alhajjar.15,Essen.93,Tan.89} .

For fixed integer $d\ge 0$, 
Alhajjar defines the invariant subring $AL_d(X)=\cap_Dk[\mathcal{F}_d]$, where $D$ ranges over all locally nilpotent derivations of $k[X]$. These rings generalize the well-known Makar-Limanov invariant of $k[X]$. 
Alhajjar's goal is to study affine $k$-varieties $X$ which are semi-rigid, meaning that $X$ admits an essentially unique non-trivial $\G_a$-action. 
In this case, $AL_d(X)=k[\mathcal{F}_d]$ for the non-trivial action, and he studies the sequence of inclusions $AL_d(X)\subset AL_{d+1}(X)$. 

\subsection{Remark} If $K$ is an algebraically closed field of positive characteristic, then any $\G_a$-action on $\A^n_K$ induces a degree function, hence a filtration, on the coordinate ring
$R=K^{[n]}$, where elements of the invariant ring $R^{\G_a}$ have degree zero. See, for example, \cite{Crachiola.Makar-Limanov.05}.
In this way, one obtains a canonical factorization of the quotient morphism for such an action, under the assumption that $R^{\G_a}$ is affine over $K$.\footnote{
It is unknown whether this assumption is necessary, i.e., it is an open question whether $R^{\G_a}$ is always finitely generated when $K$ is of positive characteristic, even in the linear case.}

\subsection{The Freeness Conjecture} 

Let $B=k^{[3]}$.  Given nonzero $D\in {\rm LND}(B)$ with $A=\krn D$, Miyanishi's Theorem asserts that $A\cong k^{[2]}$.
We make the following.
\medskip

\noindent {\bf Conjecture.} {\it Let $B=k^{[3]}$. Given $D\in {\rm LND}(B)$, if $A=\krn D$, then the following equivalent conditions hold.
\begin{itemize}
\item [1.] $B$ is a free $A$-module with basis $\{ Q_i\}_{i\ge 0}$ such that $\deg_DQ_i=i$.
\item [2.] Each degree module $\mathcal{F}_n$ is a free $A$-module with basis $\{ Q_i\}_{0\le i\le n}$ such that $\deg_DQ_i=i$.
\item [3.] Each image ideal $I_n=D^n\mathcal{F}_n\subset A$ is principal. 
\end{itemize}
}
\medskip

In addition to the examples above, evidence for this conjecture includes the following.  
\begin{itemize}
\item [(a)] 
$\mathcal{F}_1=A\oplus Ar$ for a local slice $r$. 
\item [(b)] By combining Miyanishi's Theorem with the Quillen-Suslin Theorem, Daigle has shown that $\mathcal{F}_n$ is a free $A$-module of rank $n+1$ for each $n\ge 0$ (unpublished). 
\end{itemize}

What are the geometric implications of this conjecture? Given $D\in {\rm LND}(B)$, let
\[
A=B_0\subset B_{n_1}\subset\cdots\subset B_{n_r}=B \quad {\rm and}\quad X=X_r  \xrightarrow{\pi_{r-1}} X_{r-1}  \to\cdots  \xrightarrow{\pi_1} X_1\xrightarrow{\pi_0} X_0=Y
\]
be the degree resolution of $B$ and canonical factorization of $\pi$, respectively, induced by $D$. 
If the Freeness Conjecture holds for $D$, then $B_{n_i}=B_{n_{i-1}}[Q_{n_i}]$ for each $i$, $2\le i\le r$, meaning that each mapping $\pi_i :X_{i+1}\to X_i$ is a {\it principal} affine modification 
($1\le i\le r-1$). 
By {\it Lemma\,\ref{principal}}, the defining relation for this extension is of the form $f^mQ_{n_i}=g$ for some $f\in A$, $g\in B_{n_{i-1}}$ and $m\ge 1$. 

Moreover, observe that, in the examples of {\it Sect.\,\ref{examples}}, we saw that $B_{n_{i-1}}=k[x_1,x_2,x_3,x_4]$ for some $x_j\in B$, where $x_1\in A$ and $x_1Q_{n_i}=x_4$.
Consequently $B_{n_i}=[x_1,x_2,x_3,Q_{n_i}]$, meaning that, for these examples, each intermediate threefold $X_i$ is a hypersurface in $\A^4$. 

\bibliography{bibfile}

\begin{thebibliography}{10}

\bibitem{Alhajjar.15}
B.~Alhajjar, \emph{Locally nilpotent derivations of integral domains}, Ph.D.
  thesis, Universite de Bourgogne, 2015.

\bibitem{Crachiola.Makar-Limanov.05}
A.~Crachiola and L.~Makar-Limanov, \emph{On the rigidity of small domains}, J.
  Algebra \textbf{284} (2005), 1--12.

\bibitem{Davis.67}
E.D. Davis, \emph{Ideals of the principal class, {$R$}-sequences and a certain
  monodical transformation}, Pacific J. Math. \textbf{20} (1967), 197--205.

\bibitem{Essen.93}
A.~van~den Essen, \emph{An algorithm to compute the invariant ring of a
  {${\mathbb G}_a$}-action on an affine variety}, J.\ Symbolic Comp.
  \textbf{16} (1993), 551--555.

\bibitem{Flenner.Zaidenberg.05}
H.~Flenner and M.~Zaidenberg, \emph{Locally nilpotent derivations on affine
  surfaces with a {${\mathbb C}^*$}-action}, Osaka J. Math. \textbf{42} (2005),
  931--974.

\bibitem{Freudenburg.06}
G.~Freudenburg, \emph{Algebraic {T}heory of {L}ocally {N}ilpotent
  {D}erivations}, Encyclop{\ae}dia of Mathematical Sciences, vol. 136,
  Springer-Verlag, Berlin, Heidelberg, New York, 2006.

\bibitem{Freudenburg.09}
\bysame, \emph{Derivations of ${R}[{X},{Y},{Z}]$ with a slice}, J. Algebra
  \textbf{322} (2009), 3078--3087.

\bibitem{Kaliman.94}
S.~Kaliman, \emph{Exotic analytic structures and {E}isenman intrinsic
  measures}, Israel Math. J. \textbf{88} (1994), 411--423.

\bibitem{Kaliman.04}
\bysame, \emph{Free {${\mathbb C}^+$}-actions on {${\mathbb C}^3$} are
  translations}, Invent. Math. \textbf{156} (2004), 163--173.

\bibitem{Kaliman.09}
\bysame, \emph{Actions of {${\mathbb C}^{\ast}$} and {${\mathbb C}_+$} on
  affine algebraic varieties}, Proc. Sympos. Pure Math. \textbf{80} (2009),
  629--654.

\bibitem{Kaliman.Zaidenberg.99}
S.~Kaliman and M.~Zaidenberg, \emph{Affine modifications and affine
  hypersurfaces with a very transitive automorphism group}, Transform. Groups
  \textbf{4} (1999), 53--95.

\bibitem{Rentschler.68}
R.~Rentschler, \emph{Op\'erations du groupe additif sur le plan affine}, C.\
  R.\ Acad.\ Sc.\ Paris \textbf{267} (1968), 384--387.

\bibitem{Tan.89}
L.~Tan, \emph{An algorithm for explicit generators of the invariants of the
  basic {${\mathbb G}_a$}{-}actions}, Comm.\ Algebra \textbf{17} (1989),
  565--572.

\bibitem{Winkelmann.90}
J.~Winkelmann, \emph{On free holomorphic {${\mathbb C}$}-actions on {${\mathbb
  C}^n$} and homogeneous {S}tein manifolds}, Math. Ann. \textbf{286} (1990),
  593--612.

\bibitem{Winkelmann.03}
\bysame, \emph{Invariant rings and quasiaffine quotients}, Math. Z.
  \textbf{244} (2003), 163--174.

\bibitem{Zariski.43}
O.~Zariski, \emph{Foundations of a general theory of birational
  correspondences}, Trans. Amer. Math. Soc. \textbf{53} (1943), 497--542.

\end{thebibliography}
\bibliographystyle{amsplain}


\bigskip

\noindent \address{Department of Mathematics\\
Western Michigan University\\
Kalamazoo, Michigan 49008}\\
\email{gene.freudenburg@wmich.edu}
\bigskip

\end{document}